\documentclass[12pt]{article}
\usepackage[plain]{fullpage}
\usepackage{amsmath}
\usepackage{graphicx}
\usepackage{amssymb}
\usepackage{url}
\usepackage{bbm}
\usepackage{epsfig}
\usepackage{sectsty}
\usepackage{nicefrac}
\usepackage{enumitem}
\usepackage{algorithm}
\usepackage{algorithmic}
\usepackage{array}
\usepackage{caption}
\usepackage{arydshln}
\usepackage{multirow}
\usepackage{color}

\usepackage[text={16cm,24cm}]{geometry}

\newenvironment{proof}{{\bf Proof}:\ }%
   {~\ \hfill $\Box$\vspace{0,5cm}}

    {~\ \hfill$\Diamond$\vspace{0,5cm}}

\newtheorem{theorem}{Theorem}
\newtheorem{conjecture}{Conjecture}

\newtheorem{lemma}[theorem]{Lemma}

\numberwithin{equation}{section}

\begin{document}

\newcommand\thetitle{On graceful difference labelings\\of  disjoint unions of circuits}

\title{\textbf{\thetitle}}
\author{
A.\  Hertz \footnotemark[1]
\and
C.\  Picouleau \footnotemark[2]
}
\date{ }

\def\thefootnote{\fnsymbol{footnote}}

\footnotetext[1]{ \noindent
Polytechnique Montr\'eal and GERAD, Montr\'eal (Canada). Email: {\tt alain.hertz@polymtl.ca}
}
\footnotetext[2]{ \noindent
CEDRIC - CNAM, Paris (France). Email: {\tt christophe.picouleau@cnam.fr}
}

\graphicspath{{.}{graphics/}}

\maketitle

\begin{abstract}
A graceful difference labeling (gdl for short) of a directed graph $G$ with vertex set $V$ is a bijection $f:V\rightarrow\{1,\ldots,\vert V\vert\}$ such that, when each arc $uv$ is assigned the difference label $f(v)-f(u)$, the resulting arc labels are distinct. We conjecture that all disjoint unions of circuits have a gdl, except in two particular cases. We prove partial results which support this conjecture. 
\end{abstract}

\vspace{0.2cm}
\vspace{0.2cm}
\begin{center}
\today
\end{center}
\parindent=0cm

\setlength{\parskip}{-0.1cm}


\section{Introduction}
\label{sec:intro}
A graph labeling is the assignment of
	labels, traditionally represented by integers, to the vertices or edges, or both, of a graph, subject to certain conditions. As mentioned in the survey by Gallian \cite{BTT11}, more than one thousand papers are devoted to this subject. Among all variations, the most popular and studied graph labelings are the $\beta$-valuations introduced by Rosa in 1966 \cite{Rosa}, and later called {\it graceful labelings} by Golomb \cite{Golomb}. Formally, given a graph $G$ with vertex set $V$ and $q$ edges, a graceful labeling of $G$ is an injection $f:V\rightarrow\{0,1,\ldots,q\}$ such that, when each edge $uv$ is assigned the label $\vert f(v)-f(u)\vert$, the resulting edge labels are distinct. In other words, the vertices are labeled using integers in $\{0,1,\ldots,q\}$, and these vertex labels induce an edge labeling from $1$ to $q$. The famous Ringel-Kotzig conjecture, also known as the graceful labeling conjecture, hypothesizes that all trees are graceful. It is the focus of many papers and is still open, even for some very restricted graph classes such that trees with 5 leaves, and trees with diameter 6. The survey by Gallian \cite{BTT11} lists several papers dealing with graceful labelings of particular classes of graphs, such that the disjoint union of cliques, the disjoint union of cycles, and the union of cycles with one common vertex.\\

For a directed graph with vertex set $V$ and $q$ edges, a graceful labeling of $G$ is an injection $f:V\rightarrow\{0,1,\ldots,q\}$  such that, when each arc (i.e., directed edge) $uv$ is assigned the label $(f(v)-f(u))\ (mod\ q+1)$, the resulting arc labels are distinct. As mentioned in \cite{BTT11} and \cite{Feng}, most  results and conjectures on graceful labelings of directed graphs concern directed cycles, the disjoint union of directed cycles, and the union of directed cycles with one common vertex or one common arc. In particular, it is proved that $n\overrightarrow{\bf C_3}$, the disjoint union of $n$ copies of the directed cycle with three vertices, has a graceful labeling only if $n$ is even. However, it is not known whether this necessary condition is also sufficient.\\

In this paper, we study {\it graceful difference labelings} of directed graphs, which are defined as follows.
A graceful difference labeling (gdl for short) of a directed graph $G=(V,A)$ is a bijection $f:V\rightarrow\{1,\ldots,\vert V\vert\}$ such that, when each arc
$uv$ is assigned the {\it difference label} $f(v)-f(u)$, the resulting arc labels are distinct. The absolute value $|f(v)-f(u)|$ is called the {\it magnitude} of arc $uv$, while $f(v)$ is the {\it vertex label} of $v$. Note that in a gdl of $G$, two arcs $uv$ and $u'v'$ may have the same magnitude $|f(v)-f(u)|=|f(v')-f(u')|$ but their difference labels must then be opposite, i.e., $f(v)-f(u)=-(f(v')-f(u'))$.\\

Given two graphs $G_i=(V_i,A_i)$ and $G_j=(V_j,A_j)$ with $V_i\cap V_j=\emptyset,$ their disjoint union, denoted $G_i+G_j$, is the graph with vertex set $V_i\cup V_j$ and arc set $A_i\cup A_j$. By $pG$ we denote the disjoint union of $p$ copies of $G$. For $k\ge 2$ we denote by $\overrightarrow{\bf C_k}$ a circuit on $k$ vertices isomorphic to the directed graph with vertex set $V=\{v_1,\ldots, v_k\}$ and arc set $A=\{v_iv_{i+1}:\ 1\le i<k\}\cup\{v_kv_1\}$. The circuit $\overrightarrow{\bf C_3}$ is also called a directed triangle, or simply a triangle. For all graph theoretical terms not defined here the reader is referred to \cite{West}.\\


Not every directed graph has a gdl. Indeed, a  necessary condition for $G=(V,A)$ to have a gdl is $\vert A\vert\le2(\vert V\vert-1)$. Nevertheless  this condition is not sufficient since, for example, $\overrightarrow{\bf C_3}$ has no gdl. Indeed, all bijections $f:V\rightarrow\{1,2,3\}$ induce two difference labels equal to 1, or two equal to -1.  As a second example, $\overrightarrow{\bf C_2}+\overrightarrow{\bf C_3}$ has no gdl. Indeed,
\begin{itemize}
	\item if the two arcs of $\overrightarrow{\bf C_2}$ have a magnitude equal to 1, 2, or 3, then $\overrightarrow{\bf C_3}$ also has an arc with the same magnitude, which means that two arcs in $\overrightarrow{\bf C_2}+\overrightarrow{\bf C_3}$ have the same difference label;
	\vspace{-0.3cm}\item if the magnitude of two arcs of $\overrightarrow{\bf C_2}$ is equal to 4, then two difference labels in $\overrightarrow{\bf C_3}$ are equal to 1 or to -1.
\end{itemize}
We conjecture that all disjoint unions of circuits have a gdl, except for the two cases mentioned above. We were not able to prove this conjecture, but give partial results on it. In particular, we show that $n\overrightarrow{\bf C_3}$ has a gdl if and only if $n\geq 2$.

\section{Partial proof of the conjecture}
We are interested in determining which disjoint unions of circuits have a gdl. As already mentioned in the previous section,  $\overrightarrow{\bf C_3}$ and $\overrightarrow{\bf C_2}+\overrightarrow{\bf C_3}$ have no gdl. We conjecture that these two graphs are the only two exceptions. As first result, we show that if $G$ is a circuit of length $k=2$ or $k\geq 4$, then $G$ has a gdl. We next prove that if $G$ has a gdl, and if $G'$ is obtained by adding to $G$ a circuit of even length $k=2$ or $k\geq6$, or two disjoint circuits of length 4, then $G'$ also has a gdl. We also show that the disjoint union of $\overrightarrow{\bf C_4}$ with a circuit of odd length has a gdl. All together, these results prove that if $G$ is  the disjoint union of circuits, among which at most one has an odd length, then $G$ has a gdl, unless $G=\overrightarrow{\bf C_3}$ or $G=\overrightarrow{\bf C_2}+\overrightarrow{\bf C_3}$.\\

We next show that the disjoint union of $n\geq 2$ circuits of length 3 has a gdl, and this is also the case if a $\overrightarrow{\bf C_4}$ is added to $n\overrightarrow{\bf C_3}$. Hence, if $G$ is the union of disjoint circuits with no odd circuit of length $k\geq 5$, then $G$ has a gdl, unless $G=\overrightarrow{\bf C_3}$ or $G=\overrightarrow{\bf C_2}+\overrightarrow{\bf C_3}$.
In order to prove the above stated conjecture, it will thus remain to show that if $G$ is the disjoint union of circuits with at least two odd circuits, among which at least one has length $k\geq 5$, then $G$ has a gdl. \\

Our first lemma shows that all circuits have a gdl, except $\overrightarrow{\bf C_3}$.

\begin{lemma}\label{Ck}
	The circuit $\overrightarrow{\bf C_k}$ with $k=2$ or $k\geq 4$ has a gdl. Moreover, if $k\geq 5$, then $\overrightarrow{\bf C_k}$  has a gdl with exactly one arc of magnitude 1.
\end{lemma}
\begin{proof}
	Clearly, $\overrightarrow{\bf C_2}$ has a gdl since the two bijections $f\ : \ V\rightarrow\{1,2\}$ have $1$ and $-1$ as difference labels. So assume $k\geq 4$. We distinguish four cases, according to the value of $k\mod4$:
	\begin{itemize}
		\item if $k=4p,p\ge 1$, we consider the following vertex labels:
		\vspace{-0.2cm}\begin{itemize}
			\item $f(v_{2i+1})=i+1, 0\le i\le 2p-2$;
			\item $f(v_{2i})=4p+1-i, 1\le i\le 2p-2$;
			\item $f(v_{4p-2})=2p+1$, $f(v_{4p-1})=2p+2$, $f(v_{4p})=2p$.
		\end{itemize}
		\vspace{-0.2cm}Clearly, $f$ is a bijection between $\{v_1,\ldots,v_k\}$ and $\{1,\ldots,k\}$ with the following difference labels:
		\vspace{-0.2cm}\begin{itemize}
			\item $f(v_{i+1})-f(v_i)=(-1)^{i+1}(4p-i),1\le i\le 4p-4$;
			\item $f(v_{4p-2})\!-\!f(v_{4p-3})\!=\!2$, $f(v_{4p-1})\!-\!f(v_{4p-2})\!=\!1$, $f(v_{4p})\!-\!f(v_{4p-1})\!=\!-2$, $f(v_{1})\!-\!f(v_{4p})\!=\!-2p+1$.
		\end{itemize}
	\vspace{-0.2cm}All magnitudes are distinct, except in three cases:
	\vspace{-0.2cm}\begin{itemize}
		\item $f(v_{4p-2})-f(v_{4p-3})=2$ and $f(v_{4p})-f(v_{4p-1})=-2$;
		\item for $p\geq 3$, $f(v_{2p+2})-f(v_{2p+1})=2p-1$ and $f(v_{1})-f(v_{4p})=-(2p-1)$;
		\item for $p=1$, $f(v_{4p-1})-f(v_{4p-2})=1$ and $f(v_{1})-f(v_{4p})=-1$.
	\end{itemize}
	\vspace{-0.2cm}Hence, $f$ is a gdl, and there is exactly one arc of magnitude 1 when $p\geq 2$.
	
	\item if $k=4p+1,p\ge 1$, we consider the following vertex labels:
\vspace{-0.2cm}	\begin{itemize}
		\item $f(v_{2i+1})=i+1, 0\le i\le 2p$;
		\item $f(v_{2i})=4p+2-i, 1\le i\le 2p$.
	\end{itemize}
	\vspace{-0.2cm}	Again, $f$ is a bijection between $\{v_1,\ldots,v_k\}$ and $\{1,\ldots,k\}$ with the following difference labels: 
	\vspace{-0.2cm}\begin{itemize}
		\item $f(v_{i+1})-f(v_i)=(-1)^{i+1}(4p+1-i),1\le i\le 4p$;
		\item $f(v_{1})-f(v_{4p+1})=-2p$.
	\end{itemize}
		\vspace{-0.2cm}All magnitudes are distinct, except for one pair of arcs : $f(v_{2p+2})-f(v_{2p+1})=2p$ and $f(v_{1})-f(v_{4p+1})=-2p$. 	Hence, $f$ is a gdl with exactly one arc of magnitude 1.
	
	\item if $k=4p+2,p\ge 0$, we consider the following vertex labels:
	\vspace{-0.2cm}\begin{itemize}
		\item $f(v_{2i+1})=i+1, 0\le i\le 2p$;
		\item $f(v_{2i})=4p+3-i, 1\le i\le 2p+1$.
	\end{itemize}
	\vspace{-0.2cm}Here also, $f$ is a bijection between $\{v_1,\ldots,v_k\}$ and $\{1,\ldots,k\}$ with the following difference labels: 
	\vspace{-0.2cm}\begin{itemize}
		\item $f(v_{i+1})-f(v_i)=(-1)^{i+1}(4p+2-i),1\le i\le 4p+1$;
		\item $f(v_{1})-f(v_{4p+2})=-2p-1$.
	\end{itemize}
	\vspace{-0.2cm}There are only two equal magnitudes : $f(v_{2p+2})-f(v_{2p+1})=2p+1$ and $f(v_{1})-f(v_{4p+2})=-(2p+1)$. 	Hence, $f$ is a gdl with exactly one arc of magnitude 1 when $p\geq 1$.
	
	\item if $k=4p+3,p\ge 1$, we consider the following vertex labels:
	\vspace{-0.2cm}\begin{itemize}
		\item $f(v_{2i+1})=i+1, 0\le i\le 2p-1$;
		\item $f(v_{2i})=4p+4-i, 1\le i\le 2p$;
		\item $f(v_{4p+1})=2p+2$, $f(v_{4p+2})=2p+1$, $f(v_{4p+3})=2p+3$.
	\end{itemize}
	\vspace{-0.2cm}For this last case, $f$ is a bijection between $\{v_1,\ldots,v_k\}$ and $\{1,\ldots,k\}$ with the following difference labels: 
	\vspace{-0.2cm}\begin{itemize}
		\item $f(v_{i+1})-f(v_i)=(-1)^{i+1}(4p+3-i),1\le i\le 4p-1$;
		\item $f(v_{4p+1})\!-\!f(v_{4p})\!=\!-2$, $f(v_{4p+2})\!-\!f(v_{4p+1})\!=\!-1$, $f(v_{4p+3})\!-\!f(v_{4p+2})\!=\!2$, $f(v_{1})\!-\!f(v_{4p+3})\!=\!-(2p+2)$.
	\end{itemize}
	\vspace{-0.2cm}All magnitudes are distinct, except in two cases:
	\vspace{-0.2cm}\begin{itemize}
		\item $f(v_{4p-2})-f(v_{4p-3})=2$ and $f(v_{4p})-f(v_{4p-1})=-2$;
		\item $f(v_{2p+2})-f(v_{2p+1})=2p+2$ and $f(v_{1})-f(v_{4p+3})=-(2p+2)$. 	
	\end{itemize}
	Hence, $f$ is a gdl with exactly one arc of magnitude 1.
	\end{itemize}
\vspace{-0.6cm}\end{proof}

We now show how to add two circuits of length 4, or one even circuit of length $k\geq 6$ to a graph that has a gdl.

\begin{lemma}\label{2C4}
	If a graph $G$ has a gdl, then $G+2\overrightarrow{\bf C_{4}}$ also has a gdl.
\end{lemma}
\begin{proof}
Let $\{v_1,v_2,v_3,v_4\}$ be the vertex set of the first $\overrightarrow {\bf C_{4}}$, and let $\{v_1v_2,v_2v_3,v_3v_4,v_4v_1\}$ be its arc set. Also, let $\{v_5,v_6,v_7,v_8\}$ be the vertex set of the second $\overrightarrow{\bf C_{4}}$, and let $\{v_5v_6,v_6v_7,v_7v_8,v_8v_5\}$ be its arc set.
Suppose $G=(V,A)$ has a gdl $f$. Define $f'(v)=f(v)+4$ for all $v\in V$ as well as $f'(v_1)=1, f'(v_2)=\vert V\vert+8, f'(v_3)=2, f'(v_4)=\vert V\vert+6, f'(v_5)=3, f'(v_6)=\vert V\vert+5, f'(v_7)=4,$ and $f'(v_8)=\vert V\vert+7$. Clearly, $f'$ is a bijection between $V\cup\{v_1,\ldots,v_8\}$ and $\{1,\ldots,\vert V\vert +8\}$. Moreover, the difference labels on the arcs of the two circuits are $f'(v_2)-f'(v_1)=\vert V\vert+7, f'(v_3)-f'(v_2)=-(\vert V\vert+6), f'(v_4)-f'(v_3)=\vert V\vert+4, f'(v_1)-f'(v_4)=-(\vert V\vert+5), f'(v_6)-f'(v_5)=\vert V\vert+2, f'(v_7)-f'(v_6)=-(\vert V\vert+1), f'(v_8)-f'(v_7)=\vert V\vert+3,$ and $f'(v_5)-f'(v_8)=-(\vert V\vert+4)$. Since all magnitudes in $G$ are at most equal to $\vert V\vert-1$, $f'$ is a gdl for $G+2\overrightarrow{\bf C_{4}}$.
\end{proof}

Note that in the proof of Lemma \ref{2C4}, $G$ can be the empty graph $G$ with no vertex and no arc. Hence $2\overrightarrow{\bf C_{4}}$ has a gdl.

\begin{lemma}\label{C2k}
	If a graph $G$ has a gdl, then $G+\overrightarrow{\bf C_{2k}}$ also has a gdl for  $k\ge 1,k\ne 2$.
\end{lemma}
\begin{proof}
Suppose $G=(V,A)$ has a gdl $f$, and let $\{v_1,\ldots,v_{2k}\}$ be the vertex set and $\{v_1v_2,\ldots,v_{2k-1}v_{2k},v_{2k}v_1\}$ be the arc set of $\overrightarrow{\bf C_{2k}}$. We consider two case.
\begin{itemize}
\item If $k$ is odd, then define $f'(v)=f(v)+k$ for all $v\in V$, as well as $f'(v_{2i-1})=k-i+1$ and $f'(v_{2i})=\vert V\vert+k+i$ for $1\le i\le k$. Clearly, $f'$ is a bijection between $V\cup\{v_1,\ldots,v_{2k}\}$ and $\{1,\ldots,\vert V\vert +2k\}$. Moreover, the magnitudes on $\overrightarrow{\bf C_{2k}}$ are all striclty larger than $\vert V\vert$ and all different, except in one case : $f'(v_{k+1})-f'(v_k)=\vert V\vert+k$ and $f'(v_{1})-f'(v_{2k})=-(\vert V\vert+k)$. Since all magnitudes in $G$ are strictly smaller than $\vert V\vert$, $f'$ is a gdl for $G+\overrightarrow{\bf C_{2k}}$.
\item If $k$ is even and at least equal to $4$, then set $f'(v)=f(v)+k$ for all $v\in V$, and define the vertex labels on $\overrightarrow{\bf C_{2k}}$ as follows:
\vspace{-0.2cm}\begin{itemize}
	\item $f'(v_{2i-1})=k-i+1$ for $1\le i\le k$;
	\item $f'(v_{2i})=\vert V\vert+k+i$ for $1\le i\le k-3$;
	\item $f'(v_{2k-4})=\vert V\vert+2k, f'(v_{2k-2})=\vert V\vert+2k-2, f'(v_{2k})=\vert V\vert+2k-1$.
\end{itemize}
$f'$ is bijection between $V\cup\{v_1,\ldots,v_{2k}\}$ and $\{1,\ldots,\vert V\vert +2k\}$, and all magnitudes on $\overrightarrow{\bf C_{2k}}$ are strictly larger than $\vert V\vert$. Moreover, all magnitudes on $\overrightarrow{\bf C_{2k}}$ are different, except in two cases :
\vspace{-0.2cm}\begin{itemize}
	\item $f'(v_{k})-f'(v_{k-1})=\vert V\vert+k-1$ and $f'(v_{1})-f'(v_{2k})=-(\vert V\vert+k-1)$;
	\item $f'(v_{2k-4})-f'(v_{2k-5})=\vert V\vert+2k-3$ and $f'(v_{2k-1})-f'(v_{2k-2})=-(\vert V\vert+2k-3)$.
\end{itemize} Since all magnitudes in $G$ are strictly smaller than $\vert V\vert$, $f'$ is a gdl for $G+\overrightarrow{\bf C_{2k}}$. 
\end{itemize}
\vspace{-0.5cm}\end{proof}

Since graph $G$ in the statement of Lemma \ref{2C4} is possibly empty, it follows from Lemmas \ref{Ck}, \ref{2C4} and \ref{C2k} that all disjoint unions of circuits of even length have a gdl.
We now consider disjoint unions of circuits among which exactly one has as an odd length. As already observed, $\overrightarrow{\bf C_3}$ and $\overrightarrow{\bf C_2}+\overrightarrow{\bf C_3}$ have no gdl. We show that these are the only two exceptions. According to Lemmas \ref{2C4} and \ref{C2k}, it is sufficient to prove that $2\overrightarrow{\bf C_2}+\overrightarrow{\bf C_3}$, $\overrightarrow{\bf C_4}+\overrightarrow{\bf C_{2k+1}}$ ($k\geq 1$), and $\overrightarrow{\bf C_{2k}}+\overrightarrow{\bf C_3}$ ($k\geq 3$) have a gdl.
 
\begin{lemma}\label{C2C3}
	$2\overrightarrow{\bf C_2}+\overrightarrow{\bf C_3}$ has a gdl. 
\end{lemma}
\begin{proof}
	Let $\{v_1,\ldots,v_7\}$ be the vertex set and $\{v_1v_2$, $v_2v_1$, $v_3v_4$, $v_4v_3$, $v_5v_6$, $v_6v_7$, $v_7v_5\}$ be the arc set of $2\overrightarrow{\bf C_2}+\overrightarrow{\bf C_3}$. By considering the vertex labels $f(v_1)=1$, $f(v_2)=6$, $f(v_3)=3$, $f(v_4)=7$, $f(v_5)=2$, $f(v_6)=4$ and $f(v_7)=5$, it is easy to observe that $f$ is a gdl. 
\end{proof}

\begin{lemma}\label{C2k+1C4}
	$\overrightarrow{\bf C_4}+\overrightarrow{\bf C_{2k+1}}$ has a gdl for every $k\ge 1$.
\end{lemma}
\begin{proof}
	Let $G=\overrightarrow{\bf C_4}+\overrightarrow{\bf C_{2k+1}}$. We distinguish two cases:
	\vspace{-0.2cm}\begin{itemize}
	\item if $k$ is odd, then $G$ contains $n=4(\frac{k+1}{2})+3$ vertices. Consider the vertex labels of $\overrightarrow{\bf C_{n}}$ used in the last case of the proof of Lemma \ref{Ck}, with $p=\frac{k+1}{2}$, and assume that $\{v_1,v_{n-2},v_{n-1},v_n\}$ is the vertex set of the $\overrightarrow{\bf C_4}$ in $G$, while $\{v_2,v_{3},\ldots,v_{n-3}\}$ is the vertex set of the $\overrightarrow{\bf C_{2k+1}}$. It is sufficient to prove that the difference labels on $v_1v_{n-2}$ and $v_{n-3}v_2$ do not appear on any other arc of $G$.
	\vspace{-0.2cm}\begin{itemize}
		\item $f(v_{n-2})-f(v_1)=(2p+2)-1=(k+3)-1=k+2$, which is an odd positive number, while all other odd difference labels are negative.
		\item $f(v_2)-f(v_{n-3})=(4p+3)-(2p+4)=2p-1=k$, which is again an odd positive number, different for the other negative odd labels.
	\end{itemize}
	\item if $k$ is even, consider the vertex labels of  $\overrightarrow{\bf C_{2k+4}}$ used in the first case of the proof of Lemma \ref{Ck} with $p=\frac{k}{2}+1\geq 2$ (i.e., $4p=2k+4$). Also, define $f(v_{2k+5})=2k+5=4p+1$. Assume that $\{v_1,v_{2k+2},v_{2k+3},v_{2k+4}\}$ is the vertex set of the $\overrightarrow{\bf C_4}$ in $G$, while $\{v_2,v_{3},\ldots,v_{2k+1},v_{2k+5}\}$ is the vertex set of the $\overrightarrow{\bf C_{2k+1}}$. It is sufficient to prove that the difference labels on $v_1v_{2k+2}$, $v_{2k+5}v_2$, and $v_{2k+1}v_{2k+5}$ do not appear on any other arc of $G$.
	\vspace{-0.2cm}\begin{itemize}
	\item $f(v_{2k+2})-f(v_1)=(2p+1)-1=(k+3)-1=k+2$, which is an even positive number, while all other even difference labels are negative.
	\item $f(v_2)-f(v_{2k+5})=(4p)-(4p+1)=-1$. Since $p>1$, the only other arc with magnitude 1 is $v_{2k+2}v_{2k+3}$ which has a difference label of 1.
	\item $f(v_{2k+5})-f(v_{2k+1})=(4p+1)-(2p-1)=2p+2=k+4$, which is again an even positive number, while all other even difference labels are negative.
		\end{itemize}
		\end{itemize}
\vspace{-0.6cm}		\end{proof}

\begin{lemma}\label{C3Ck}
	$\overrightarrow{\bf C_k}+\overrightarrow{\bf C_3}$ has a gdl for every $k\ge 5$.
\end{lemma}
\begin{proof}
	Let $\{v_1,\ldots,v_{k+3}\}$ be the vertex set and $\{v_1v_2, \ldots, v_{k-1}v_{k}$, $v_kv_1$, $v_{k+1}v_{k+2}$, $v_{k+2}v_{k+3}$, $v_{k+3}v_{k+1} \}$ be the arc set of $G=\overrightarrow{\bf C_k}+\overrightarrow{\bf C_3}$. Consider the gdl $f$ defined in the proof of Lemma \ref{Ck} for $\overrightarrow{\bf C_k}$, and set 
	$f'(v_i)=f(v_i)+2$ for all $i=1,\ldots,k$. If the only arc of magnitude 1 has a difference label equal to -1, then define  
	$f'(v_{k+1})=1$, $f'(v_{k+2})=2$, and $f'(v_{k+3})=k+3$, else define
	$f'(v_{k+1})=2$, $f'(v_{k+2})=1$, and $f'(v_{k+3})=k+3$. Clearly, $f'$ is a bijection between $\{v_1,\ldots,v_{k+3}\}$ and $\{1,\ldots,k+3\}$. To conclude that $f'$ is a gdl, it is sufficient to prove that the difference labels on $\overrightarrow{\bf C_3}$ do not appear on $\overrightarrow{\bf C_k}$.
	\begin{itemize}
		\item The arc $v_{k+1}v_{k+2}$ is of magnitude 1, and its difference label has the sign opposite to that of magnitude 1 in  $\overrightarrow{\bf C_k}$;
		\item The magnitudes of $v_{k+2}v_{k+3}$ and $v_{k+3}v_{k+1}$ are distinct and larger than $k$, while all magnitudes in $\overrightarrow{\bf C_k}$ are strictly smaller than $k$.
	\end{itemize} 
\vspace{-0.4cm}\end{proof}

All together, the previous lemmas show that if $G$ be the disjoint union of circuits, among which at most one has an odd length, then $G$ has a gdl if and only if $G \neq \overrightarrow{\bf C_3}$ and $G\neq  \overrightarrow{\bf C_2}+\overrightarrow{\bf C_3}$.
We now consider the disjoint union of $n$ circuits of length 3, and show that these graphs have a gdl for all $n\geq 2$.\\

\begin{lemma}\label{nC3}
	For every $n\geq 2$, the graph $n\overrightarrow{\bf C_3}$ has a gdl with at most one arc of magnitude $3n-2$, and all other arcs of magnitude strictly smaller than $3n-2$.
\end{lemma}
\begin{proof}
	The graphs in Figures \ref{2C3},\ldots,  \ref{9C3}  show the existence of the desired gdl for $2\leq n \leq 9$.
\begin{figure}[H]
	\centering
	\includegraphics[width=8cm,keepaspectratio=true]{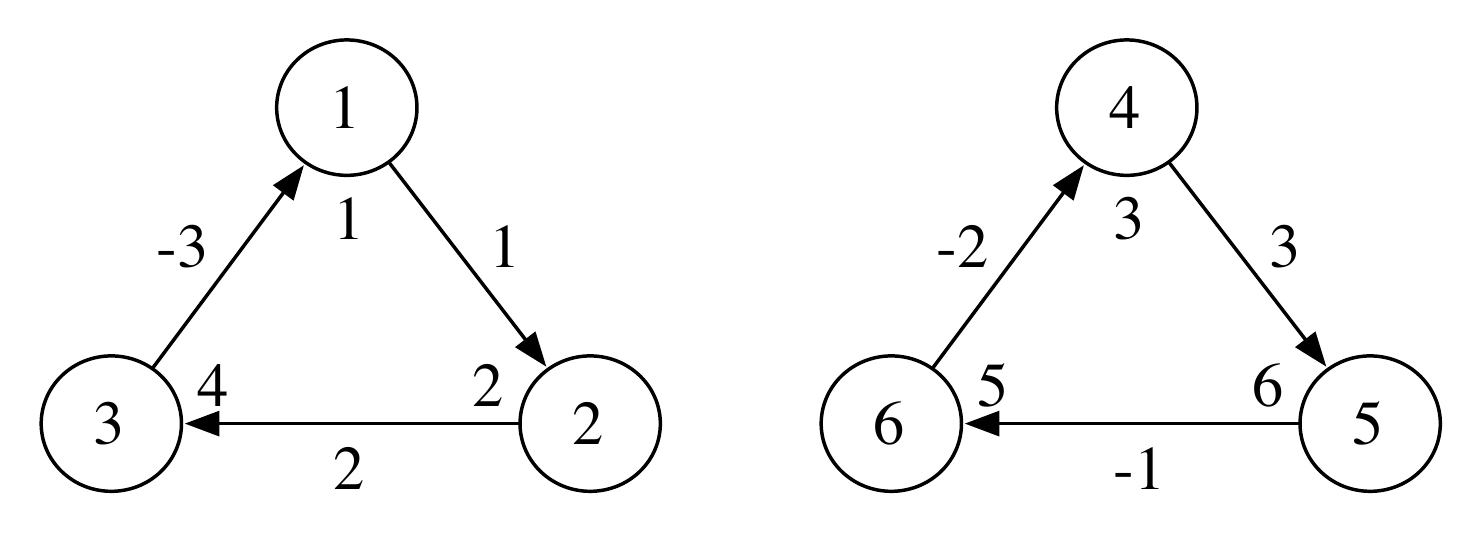}
	\caption{$2\protect\overrightarrow{\bf C_3}$.}	
	\label{2C3}
\end{figure}

\vspace{0cm}\begin{figure}[H]
	\centering
	\includegraphics[width=12cm,keepaspectratio=true]{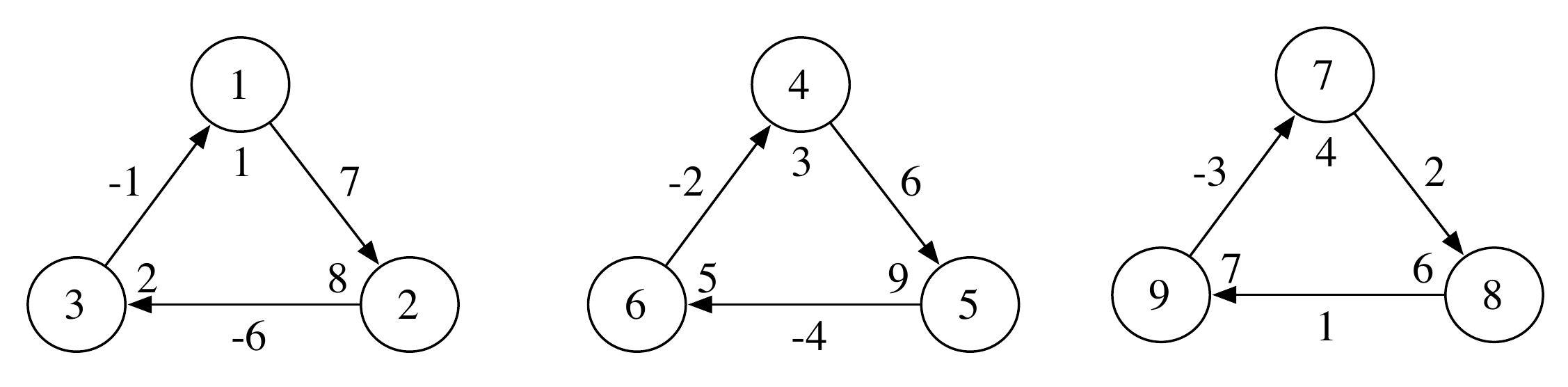}
	\caption{$3\protect\overrightarrow{\bf C_3}$.}	
	\label{3C3}
\end{figure}

\vspace{0cm}\begin{figure}[H]
	\centering
	\includegraphics[width=16cm,keepaspectratio=true]{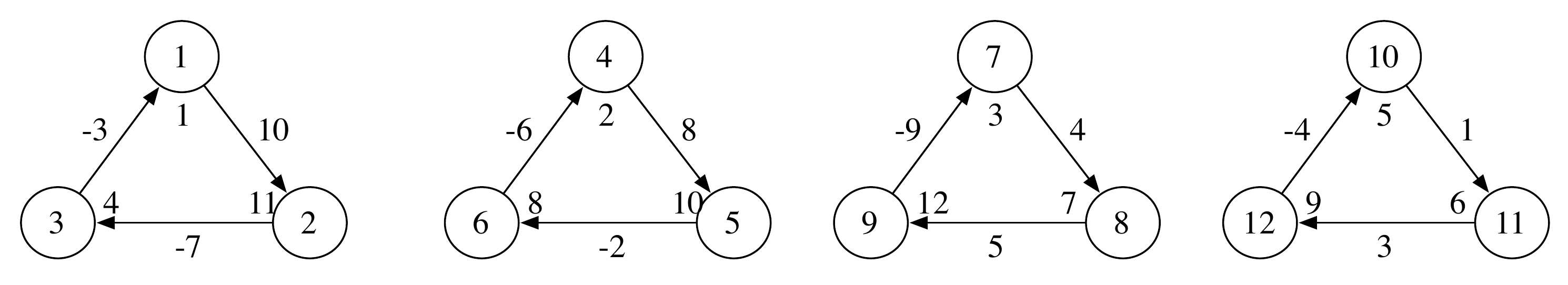}
	\caption{$4\protect\overrightarrow{\bf C_3}$.}	
	\label{4C3}
\end{figure}

\vspace{0cm}\begin{figure}[H]
	\centering
	\includegraphics[width=12cm,keepaspectratio=true]{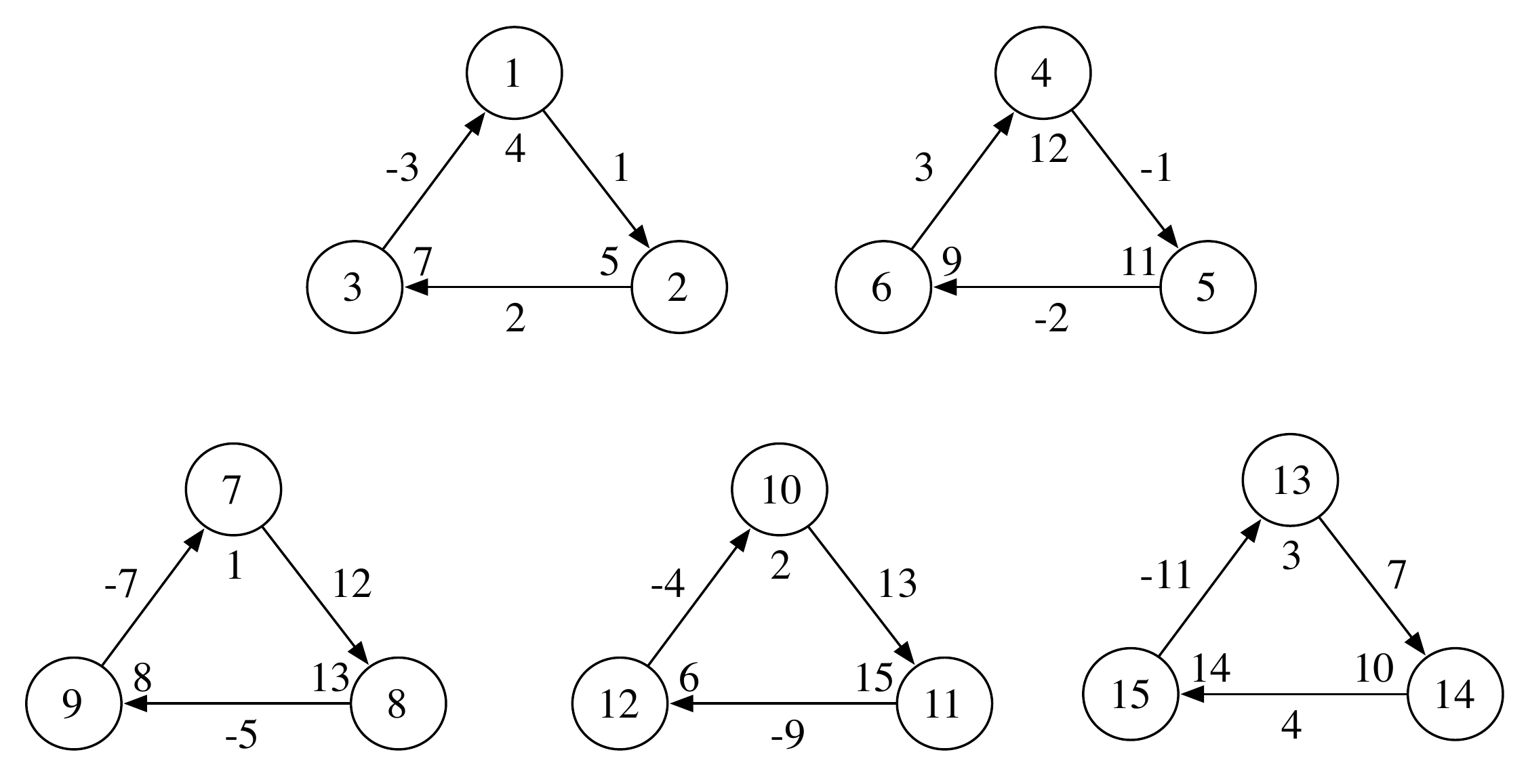}
	\caption{$5\protect\overrightarrow{\bf C_3}$.}	
	\label{5C3}
\end{figure}

\vspace{0cm}\begin{figure}[H]
	\centering
	\includegraphics[width=12cm,keepaspectratio=true]{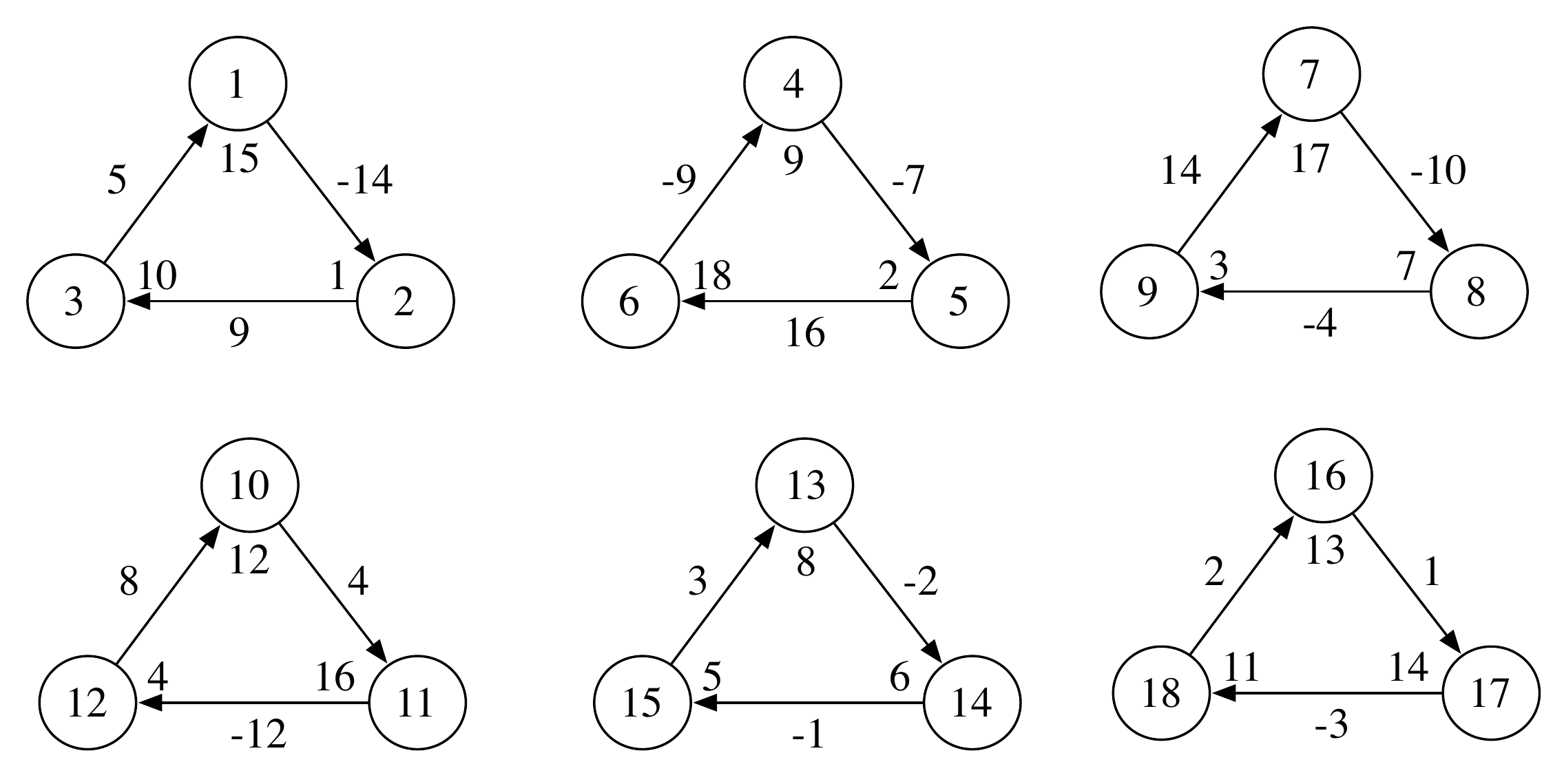}
	\caption{$6\protect\overrightarrow{\bf C_3}$.}	
	\label{6C3}
\end{figure}

\vspace{0cm}\begin{figure}[H]
	\centering
	\includegraphics[width=16cm,keepaspectratio=true]{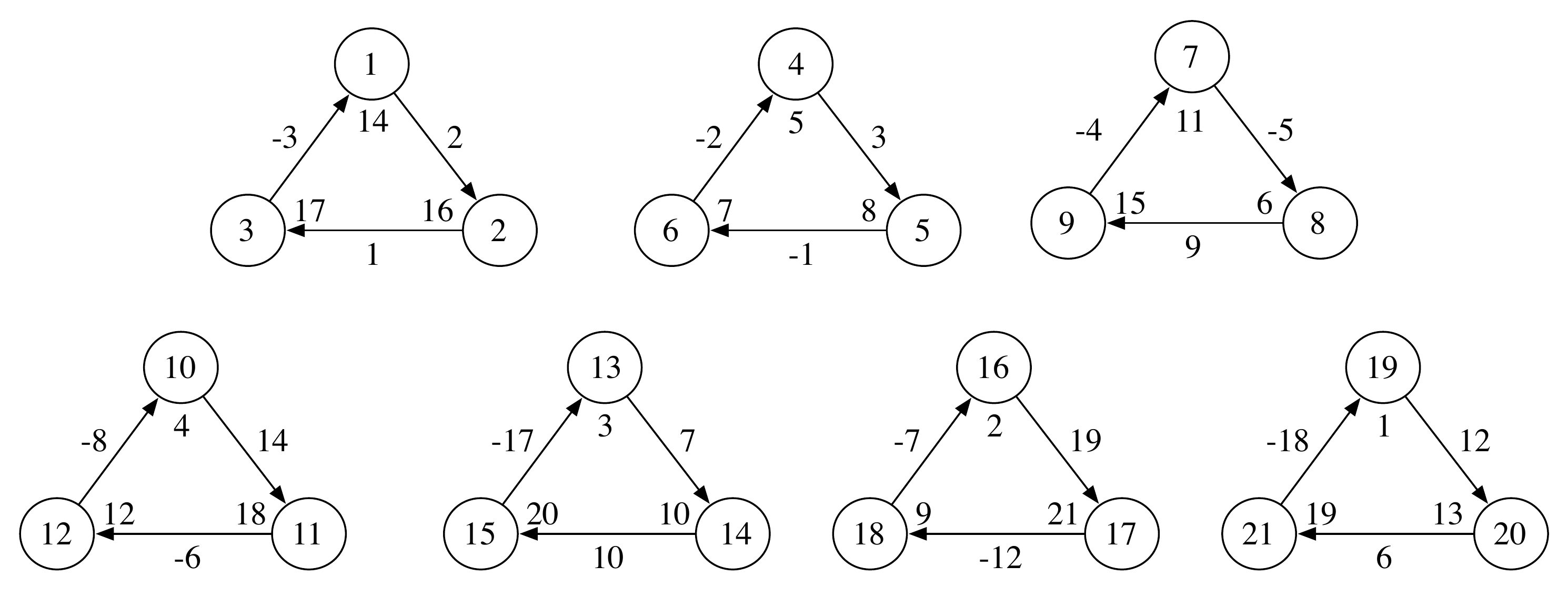}
	\caption{$7\protect\overrightarrow{\bf C_3}$.}	
	\label{7C3}
\end{figure}

\vspace{0cm}\begin{figure}[H]
	\centering
	\includegraphics[width=16cm,keepaspectratio=true]{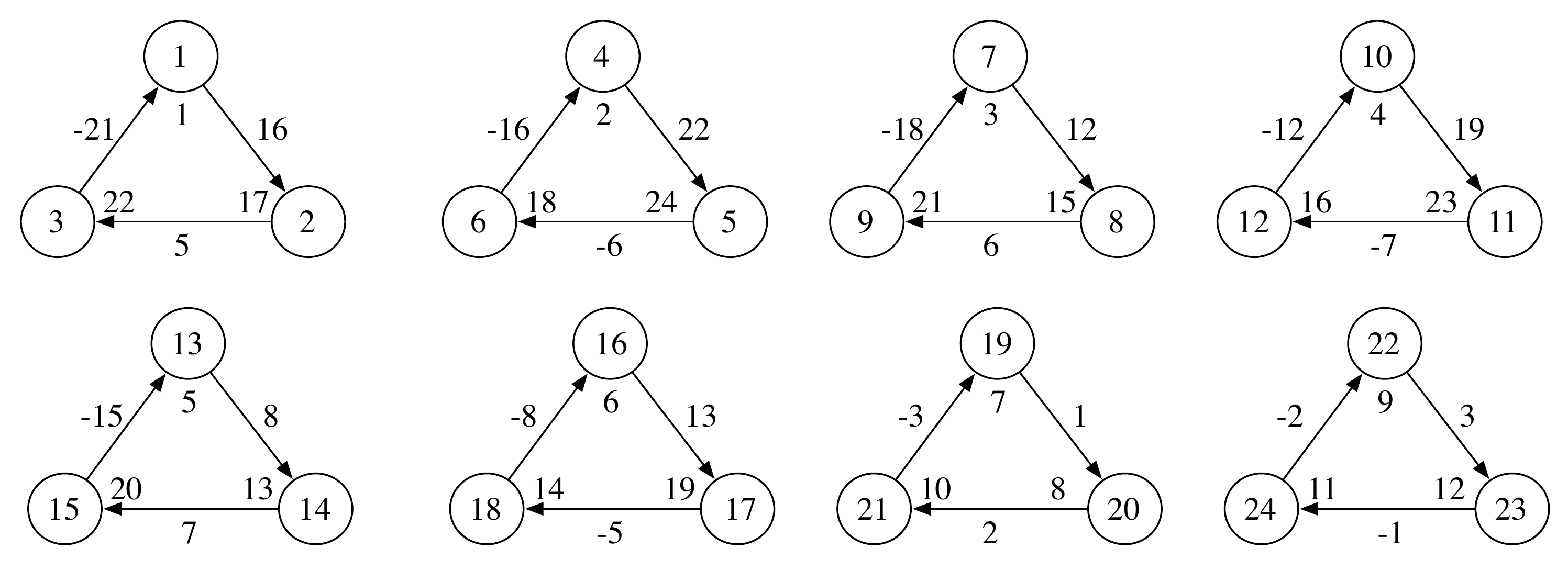}
	\caption{$8\protect\overrightarrow{\bf C_3}$.}	
	\label{8C3}
\end{figure}

\vspace{0cm}\begin{figure}[H]
	\centering
	\includegraphics[width=12cm,keepaspectratio=true]{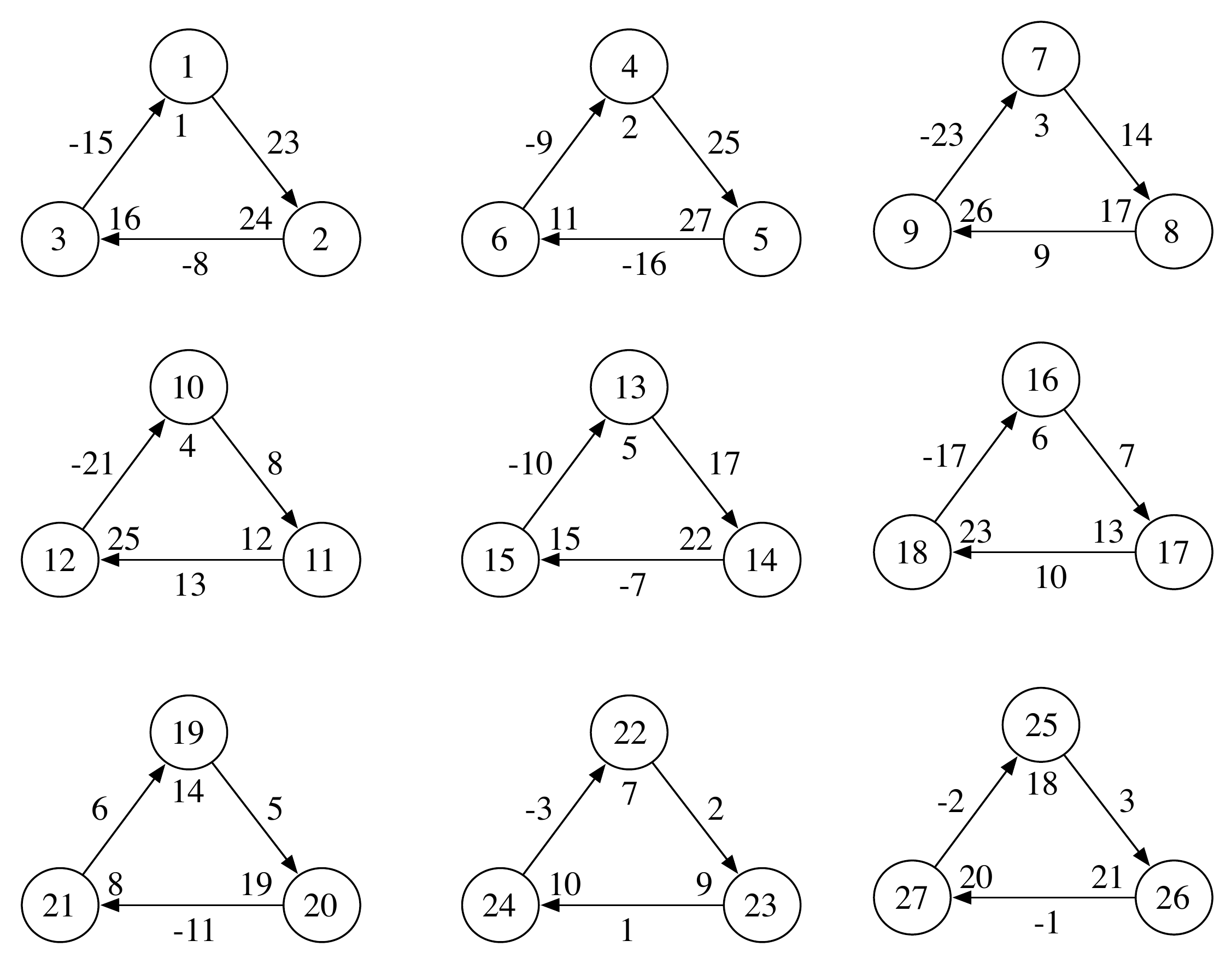}
	\caption{$9\protect\overrightarrow{\bf C_3}$.}	
	\label{9C3}
\end{figure}

	 We now prove the result by induction on $n$. So, consider the graph $n\overrightarrow{\bf C_3}$ with $n\geq 10$, and assume the result is true for less than $n$ directed triangles. Let $t$ and $r$ be two integers such that $-4\leq r \leq 2$ and
	 $$n=7t+r.$$

We thus have $t\geq 2$. We will show how to construct a gdl for $n\overrightarrow{\bf C_3}$ given a gdl for $t\overrightarrow{\bf C_3}$. We thus have to add $n-t$ directed triangles to $t\overrightarrow{\bf C_3}$. For this purpose, define 
		$$\theta=\left\lceil\frac{n-t}{2}\right\rceil=3t+\left\lceil\frac{r}{2}\right\rceil.$$
		
It follows that $n-t=2\theta$ if $r$ is even, and $n-t=2\theta-1$ if $r$ is odd. We now prove the lemma by considering the 4 cases A,B,C,D defined in Table \ref{tableauCas}.
		
	\setlength{\doublerulesep}{\arrayrulewidth}		
	\begin{table}[H]
				\begin{center}
				\begin{tabular}{|c|cc|c|}
						\hline
						$n-t$&$r$&$\theta$&Case\\
						\hline
 						\multirow{4}{12pt}{$2\theta$}&-4&$3t-2$&\multirow{3}{0pt}{A}\\
 						&-2&$3t-1$&\\
 						&0&$3t$&\\
 						\cline{2-4}
 						&2&$3t+1$&\multirow{1}{0pt}{B}\\
 						\hline
 						\multirow{3}{25pt}{$2\theta-1$}&-3&$3t-1$&\multirow{2}{0pt}{C}\\
 						&-1&$3t$&\\
 						\cline{2-4}
 						&1&$3t+1$&\multirow{1}{0pt}{D}\\
 						\hline
					\end{tabular}
							\caption{Four different cases}
							\label{tableauCas}
				\end{center}
			\end{table}

\vspace{-0.4cm}{\bf Case A :} $n=2\theta+t$, $\theta\in \{3t-2,3t-1,3t\}$

\vspace{0.1cm}Consider $2\theta$ directed triangles $T_1,\ldots,T_{2\theta}$, every $T_i$ having $\{v_{3i-2},v_{3i-1},v_{3i}\}$ as vertex set and $\{v_{3i-2}v_{3i-1}$, $v_{3i-1}v_{3i}$, $v_{3i}v_{3i-2}\}$ as arc set. Consider the vertex labels $f(v_i)$ for $T_1,\ldots,T_{2\theta}$ shown in Table \ref{num1}. 
	\begin{table}[H]
		\begin{center}
			\begin{tabular}{|c|c|c|c|}\hline
				Triangle $T_i$ &$f(v_{3i-2})$&$f(v_{3i-1})$&$f(v_{3i})$ \\
				\hline
				$T_1$&$1$&$2\theta+1$&$6\theta+3t-3$ \\
				$T_2$&$2$&$6\theta+3t$&$4\theta+3t$ \\
				$T_3$&$3$&$6\theta+3t-1$&$2\theta+2$ \\
				$T_4$&$4$&$4\theta+3t-1$&$6\theta+3t-2$ \\
				\hdashline
				$\vdots$&$\vdots$&$\vdots$&$\vdots$ \\
				$T_{2k-1}$&$2k-1$&$2\theta+k$&$6\theta+3t-2k+2$ \\
				$T_{2k}$&$2k$&$6\theta+3t-2k+1$&$4\theta+3t-k+1$ \\
				$k=3,\ldots,\theta$&&&\\
				$\vdots$&{$\vdots$}&{$\vdots$}&{$\vdots$} \\
				\hline
			\end{tabular}
			\caption{The labeling of $T_1,\ldots,T_{2\theta}$ for case A.}
			\label{num1}
		\end{center}
	\end{table}

\vspace{-0.4cm}Also, let $f'$ be a gdl for $t\overrightarrow{\bf C_3}$ with at most one arc of magnitude $3t-2$, and all other arcs of magnitude strictly smaller than $3t-2$. Define $f(v_i)=f'(v_i)+3\theta$ for $i=6\theta+1,\ldots,6\theta+3t$.
One can easily check that $f$ is a bijection between the vertex set $\{v_1,\ldots,v_{6\theta+3t}\}$ and $\{1,\ldots,6\theta+3t=3n\}$.\\
	
	For each $T_i$,  we define its {\it small} difference label (small-dl for short) as the minimum among $\vert f(v_{3i-1})-f(v_{3i-2})\vert$, $\vert f(v_{3i})-f(v_{3i-1})\vert$, and $\vert f(v_{3i-2})-f(v_{3i})\vert$. Similarly, the {\it big} difference label (big-dl) of $T_i$ is the maximum of these three values, and the {\it medium} one (medium-dl) is  the third value on $T_i$.
	Table \ref{dl1} gives the small, medium and big difference labels of $T_1,\ldots,T_{2\theta}$. By considering two dummy directed triangles $D_1$ and $D_2$, we have grouped the triangles into $\theta+1$ pairs $\pi_0,\ldots,\pi_{\theta}$, as shown in Table \ref{dl1}. Two triangles belong to the same pair $\pi_i$ if their small difference labels have the same magnitude. The difference labels given for $D_1$ and $D_2$ are artificial, but are helpful for simplifying the proof.
	
		\begin{table}[H]
			\begin{center}
				\begin{tabular}{|c|c|c|c|c|}\hline
					Pair&Triangle&Small-dl&Medium-dl&Big-dl \\
					\hline
					\multirow{2}{3cm}{$\pi_0=(T_1,T_2)$}&$T_1$&$2\theta$&$4\theta+3t-4$&$-(6\theta+3t-4)$ \\
					&$T_2$&$-2\theta$&$-(4\theta+3t-2)$&$6\theta+3t-2$ \\
					\hdashline
					\multirow{2}{3cm}{$\pi_1=(T_3,T_4)$}&$T_3$&$-(2\theta-1)$&$-(4\theta+3t-3)$&$6\theta+3t-4$ \\
					&$T_4$&$2\theta-1$&$4\theta+3t-5$&$-(6\theta+3t-6)$ \\
					\hdashline
					\multirow{2}{3cm}{$\pi_2=(D_1,T_5)$}&$D_1$&$-(2\theta-2)$&$-(4\theta+3t-5)$&$6\theta+3t-7$ \\
					&$T_5$&$2\theta-2$&$4\theta+3t-7$&$-(6\theta+3t-9)$ \\
					\hdashline
					$\vdots$&$\vdots$&$\vdots$&$\vdots$&$\vdots$\\
					$\pi_k=(T_{2k},T_{2k+1})$&$T_{2k}$&$-(2\theta-k)$&$-(4\theta+3t-3k+1)$&$6\theta+3t-4k+1$ \\
					$k=3,\ldots,\theta-1$&$T_{2k+1}$&$2\theta-k$&$4\theta+3t-3k-1$&$-(6\theta+3t-4k-1)$ \\
					$\vdots$&$\vdots$&$\vdots$&$\vdots$&$\vdots$\\
					\hdashline
					\multirow{2}{3cm}{$\pi_{\theta}=(T_{2\theta},D_{2})$}&$T_{2\theta}$&$-\theta$&$-(\theta+3t+1)$&$2\theta+3t+1$ \\
					&$D_{2}$&$\theta$&$\theta+3t-1$&$-(2\theta+3t-1)$ \\
					\hline
				\end{tabular}
				\caption{The difference labels of the arcs of $T_1,\ldots,T_{2\theta},D_1,D_2$ for case A.}
				\label{dl1}
			\end{center}
		\end{table}

\vspace{-0.4cm}	Let $s^1_i$ be the small-dl of the first triangle of $\pi_i$, and let $s^2_i$ be the small-dl of the its second triangle. Define $m^1_i$, $m^2_i$, $b^1_i$ and $b^2_i$ in a similar way for the medium and big difference labels of $\pi_i$. For example, $s^1_2=-(2\theta-2)$,  $s^2_2=2\theta-2$,
	 $m^1_2=-(4\theta+3t-5)$,  $m^2_2=4\theta+3t-7$,  $b^1_2=6\theta+3t-7$, and  $b^2_2=-(6\theta+3t-9)$. Note that $s^j_i+m^j_i=-b^j_i$ and $|s^j_i|+|m^j_i|=|b^j_i|$ for all $i=0,\ldots,\theta$ and $j=1,2$.
	 The following properties are valid for every $\pi_i$ with $2\leq i\leq \theta$:
	 \begin{itemize}
	 \vspace{-0.2cm}	\item $s^1_i$, $m^1_i$ and $b^2_i$ are negative integers, while $s^2_i$, $m^2_i$ and $b^1_i$ are positive integers;
	 \vspace{-0.2cm}	\item $s^2_i=-s^1_i$, $m^2_i=-m^1_i-2$, and $b^2_i=-b^1_i+2$;
	 \vspace{-0.2cm}	\item if $i<\theta$, then $s^1_{i+1}=s^1_i+1$, $m^1_{i+1}=m^1_i+3$, and $b^1_{i+1}=b^1_i-4$.
	 \end{itemize} 
	 
	 Note that all big difference labels $b^j_{i}$ have the same parity for $2\leq i\leq \theta$, $j=1,2$, while for the medium ones, the parities alternate between successive $\pi_i$ and $\pi_{i+1}$. Moreover, the largest magnitude is $6\theta+3t-2=	3n-2$, and there is exactly one arc with this magnitude. 
	 Since $\theta<3t+1$, we have $\theta+3t+1>2\theta$, which means that no medium-dl can be equal to a small-dl, with the exception of $m^2_{\theta}$ which can be equal to $2\theta$ or $2\theta-1$. But we don't care about this exception since $D_2$ (the second triangle of $\pi_{\theta}$) is a dummy triangle. Notice also that the small difference labels in Table \ref{dl1} are all distinct, which is also the case for the medium and the big ones. Since all difference labels on $T_{2\theta+1},\ldots,T_{2\theta+t}$ are distinct, we conclude that there are only two possibilities for two arcs $uv$ and $u'v'$ of $n\overrightarrow{\bf C_3}$ to have the same difference label $f(v)-f(u)=f(v')-f(u')$:
	 \begin{itemize}
	 \vspace{-0.2cm}	\item one of these arcs belongs to $T_{2\theta+1},\ldots,T_{2\theta+t}$ and the other to $T_{1},\ldots,T_{2\theta}$;
	 \vspace{-0.2cm}	\item both arcs belong to $T_{1},\ldots,T_{2\theta}$, one having a big-dl, and the other a medium-dl.
	 \end{itemize}
	 
	 Consider the first case. Remember that there is at most one arc on  $T_{2\theta+1},\ldots,T_{2\theta+t}$ with magnitude $3t-2$, all other arcs having a smaller magnitude. Since at most one arc on $T_{1},\ldots,T_{2\theta}$ has a magnitude equal to $\theta\geq 3t-2$, we conclude that such a situation can only occur at most once (with $\theta=3t-2$), and we can avoid it by flipping all triangles $T_{2\theta+1},\ldots,T_{2\theta+t}$. 
	 	More precisely, by {\it flipping} a directed triangle $\overrightarrow{C_3}$ with vertex set $\{x,y,z\}$ and arc set $\{xy,yz,zx\}$, we mean exchanging the vertex labels of $y$ and $z$. Hence, the set of difference labels is modified from $\{f(y)-f(x), f(z)-f(y), f(x)-f(z)\}$ to $\{f(z)-f(x), f(y)-f(z), f(x)-f(y2)\}$, which means that each difference label of the original set appears with an opposite sign in the modified set, but with the same magnitude.
	 \\
	 
	 Consider the second case, and let $i$ and $j$ be such that $b^x_{i}=m^y_{j}$ for $x,y$ in $\{1,2\}$.
	Note that $0\leq j < i \leq \theta$. We say that $\pi_i$ is {\it conflicting} with $\pi_j$ and we write $\pi_i\rightarrow \pi_j$. If $\pi_i$ is not conflicting with $\pi_j$, we write $\pi_i\nrightarrow \pi_j$. Note that 
	 
\begin{center}
	\hspace{1.5cm}if there are $k<j<i$ such that $\pi_i\rightarrow \pi_j\rightarrow \pi_k$, then $\pi_k\nrightarrow  \pi_{\ell}$ for all $\ell<k$. \hfill(a)
\end{center}
	 
	 Indeed, if $\pi_i\rightarrow \pi_j\rightarrow \pi_k$, then there are $x,y,z,w$ in $\{1,2\}$  such that $b^x_i=m^y_j$ and $b^z_j=m^w_k$. Then: $$|b^w_k|=|m^w_k|+|s^w_k|=|b^z_j|+|s^w_k|\geq|b^y_j|+|s^w_k|-2=|m^y_j|+|s^y_j|+|s^w_k|-2=|b^x_i|+|s^y_j|+|s^w_k|-2.$$
	 Since $|b^x_i|\geq 2\theta+3t+1,  |s^w_k |> |s^y_j|> |s^x_i|\geq\theta$, we have $\min\{|b^1_k|,|b^2_k|\}\geq|b^w_k|-2\geq4\theta+3t$.  Hence, $\pi_k\nrightarrow  \pi_{\ell}$ for all $\ell<k$ since there is no arc with medium magnitude at least equal to $4\theta+3t$.\\

	 	 We now show how to avoid conflicting pairs $\pi_i$ and $\pi_j$ with both $i$ and $j$ at least equal to 2. Conflicts involving $\pi_0$ and $\pi_1$ (i.e., $T_{1},\ldots,T_{4}$) will be handled later. Consider $i$ and $j$ such that $2\leq j<i<\theta$ and $\pi_i\rightarrow \pi_j$. Since $b^1_i$ and $m^2_j$ are positive, while $b^2_i$ and $m^1_j$ are negative, we either have $b^1_i=m^2_j$ or $b^2_i=m^1_j$. In the first case, we say that $\pi_i$ is $12-$conflicting with $\pi_j$, while in the second case, we say that  $\pi_i$ is $21-$conflicting with $\pi_j$. Note that 
	  
	  \begin{center}
	  	\hspace{0.5cm}if $\pi_i$ is $12-$conflicting with $\pi_j$, then $\pi_{i-1}$ is $21-$conflicting with $\pi_j$ and $\pi_{i+1}\nrightarrow \pi_j$. \hfill(b)
	  	
	  	\hspace{0.5cm}if $\pi_i$ is $21-$conflicting with $\pi_j$, then $\pi_{i+1}$ is $12-$conflicting with $\pi_j$ and $\pi_{i-1}\nrightarrow \pi_j$. \hfill(c)
	  \end{center}
	  
	  Indeed, if $\pi_i$ is $12-$conflicting with $\pi_j$, then $b^1_i=m^2_j$, which implies $b^2_{i-1}\!=\!-b^1_i\!-\!2\!=\!-m^2_j\!-\!2\!=\!m^1_j$. Since $\max\{|b^1_{i+1}|,|b^2_{i+1}|\}=|b^1_{i+1}|=b^1_i-4<m^2_j\leq \min\{|m^1_j|,|m^2_j|\}$, we have $\pi_{i+1}\nrightarrow \pi_j$.
	  Similarly, if
	  $\pi_i$ is $21-$conflicting with $\pi_j$, then $b^2_i=m^1_j$, which implies $b^1_{i+1}\!=\!-b^2_i\!-\!2\!=\!-m^1_j\!-\!2\!=\!m^2_j$. Moreover, since  $\min\{|b^1_{i-1}|,|b^2_{i-1}|\}=|b^2_{i-1}|=|b^2_i|+4>m^1_j= \max\{|m^1_j|,|m^2_j|\}$, we have $\pi_{i-1}\nrightarrow \pi_j$. Observe also that: 
	  \begin{center}
	  	\hspace{3cm}if $\pi_i\rightarrow \pi_j$ for $2\leq j$, then $\pi_k\nrightarrow \pi_j$ for $2\leq k\neq i,i-1,i+1$. \hfill(d)
	  \end{center}
	  Indeed, if $2\leq k<i-1$, then
	  $\min\{|b^1_k|,|b^2_k|\}\geq\max\{|m^1_j|,|m^2_j|\}+4$, while for $\theta\geq k>i+1$, we have
	  $\max\{|b^1_k|,|b^2_k|\}\leq\min\{|m^1_i|,|m^2_i|\}-4.$
	  In both cases, none of $m^1_j$ and $m^2_j$ can be equal to $b^1_k$ or $b^2_k$. As next property, note that:
	  
	  \begin{center}
	  	\hspace{5cm}if $\pi_i\rightarrow \pi_j$ for $2\leq j$, then $\pi_i\nrightarrow \pi_k$ for $1\leq k \neq j$. \hfill(e)
	  \end{center}
	Indeed, let us first show that $\pi_i\nrightarrow \pi_{j-1}$. If $j=2$, then $m^1_1\!=\!m^1_2\!-\!2\!=\!-\!m^2_2\!-\!4$ and $m^2_1\!=\!-m^1_2\!=\!m^2_2\!+\!2$. Since we have either $b^1_i=m^2_2$ and $b^2_i=-m^2_2+2$, or $b^2_i=m^1_2$ and $b^1_i=-m^1_2+2$, we see that $\pi_i\nrightarrow \pi_1$. For $j>2$, observe that $b^1_i,b^2_i,m^1_j,m^2_j$ all have the same parity, while $m^1_{j-1},m^2_{j-1}$ have the opposite parity. Hence $\pi_i\nrightarrow \pi_{j-1}$.
	  Similarly, $\pi_i\nrightarrow \pi_{j+1}$ for all $2\leq j \leq \theta-1$ since the parity of $m^1_{j+1},m^2_{j+1}$ is the opposite of the parity of $b^1_i,b^2_i$.
	  Now, let $x,y\in\{1,2\}$ be such $b^x_i=m^y_j$. If $1\leq k<j-1$, then  $\min\{|m^1_k|,|m^2_k|\}\geq\max\{|b^1_i|,|b^2_i|\}+2$,
	  while for $\theta\geq k>j+1$, $\max\{|m^1_k|,|m^2_k|\}\leq\min\{|b^1_i|,|b^2_i|\}-2$.
	  In both cases, none of $m^1_k$ and $m^2_k$ can be equal to $b^1_i$ or $b^2_i$, which proves that $\pi_i\nrightarrow \pi_{k}$ for $k\geq 1, k\neq j-1,j,j+1$.
	  \\

	 In what follows, we will remove conflicts by flipping some triangles. More precisely, by flipping $\pi_i$, we mean flipping both triangles in $\pi_i$. Note that :
	 	  \begin{center}
	 	  	\hspace{2.5cm}if $\pi_i\rightarrow \pi_j$ for $j\geq 2$, then $\pi_i\nrightarrow \pi_k$ for all $k\geq 2$ after the flip of $\pi_i$. \hfill(f)	
 \end{center}
	 	  Indeed, if $\pi_i$ is $12-$conflicting with $\pi_j$, then $b^1_i=m^2_j$, and there is no triangle with medium-dl equal to  $-b^1_i=-m^2_j=$ or $-b^2_i=b^1_i-2=m^2_j-2$. Similarly, if $\pi_i$ is $21-$conflicting with $\pi_j$, then $b^2_i=m^1_j$, and there is no triangle with medium-dl equal to $-b^1_i=-b^2_i+2=-m^1_j+2$ or $-b^2_i=-m^1_j$. Hence, we have $\pi_i\nrightarrow \pi_k$ for all $k\geq 2$ after the flip of $\pi_i$.  Also, 
	 \begin{center}
	 	  	\hspace{2.5cm}if $\pi_i\rightarrow \pi_j$ for $j\geq 2$, then $\pi_k\nrightarrow \pi_j$ for all $k\leq \theta$ after the flip of $\pi_j$. \hfill(g)	 	  		 	  	
	 \end{center}
	 	  Indeed,  if $\pi_i$ is $12-$conflicting with $\pi_j$, then $b^1_i=m^2_j$, $b^2_{i-1}=m^1_j$, and there is no triangle with a big-dl equal to $-m^1_j=-b^2_{i-1}$ or $-m^2_j=-b^1_i$. Similarly, if $\pi_i$ is $21-$conflicting with $\pi_j$, then $b^2_i=m^1_j$, $b^1_{i+1}=m^2_j$, and there is no triangle with a big-dl equal to $-m^1_j=-b^2_{i}$ or $-m^2_j=-b^1_{i+1}$. Hence, we have $\pi_k\nrightarrow \pi_j$ for all $k\leq \theta$ after the flip of $\pi_j$. \\
	 	 
	 	 Now, let $J$ be the set of integers $j$ such that $\pi_i\rightarrow \pi_j\rightarrow \pi_k$ for at least one pair $i,k$ of integers with $2\leq k<j<i\leq \theta$. Also, let $J'$ be the set of integers $j'$ such that there is $k\geq 2$ and $j\neq j'$ in $J$ with $\pi_j\rightarrow \pi_k$ and $\pi_{j'}\rightarrow \pi_k$. Note that $J\cap J'=\emptyset$. Indeed, consider $j'\in J'$, and $j\neq j'$ in $J$ such that $\pi_j\rightarrow \pi_k$ and $\pi_{j'}\rightarrow \pi_k$. It follows from (b), (c) and (d) that  $j'\!=\!j\!-\!1$ or $j'\!=\!j\!+\!1$. Since $j\in J$, $m^1_{j}$ and $m^2_{j}$ have the same parity as the big difference labels on $T_5,\ldots,T_{2\theta}$, which means that $m^1_{j'}$ and $m^2_{j'}$ have the opposite parity. Hence, there is no $i$ with $\pi_i\rightarrow \pi_{j'}$, which proves that $j'\notin J$.\\
	 	  	 
	 	 By flipping all $\pi_{\ell}$ with $\ell\in J\cup J'$, we get $\pi_i\nrightarrow \pi_j$ for all $2\leq j<i\leq \theta$ with $i$ or $j$ in $J\cup J'$. Indeed, it follows from (a) that we cannot have $\pi_i\rightarrow \pi_j$ with both $i$ and $j$ in $J\cup J'$, since this would imply the existence of $k,k'$ with $2\leq k <k'\leq \theta$ and $\pi_{k'}\rightarrow \pi_i\rightarrow \pi_j\rightarrow \pi_k$. Hence, it follows from (f) and (g) that $\pi_i\nrightarrow \pi_j$ for $i$ or $j$ in $J$, $2\leq j<i\leq \theta$. Moreover, as observed above, $j'\in J'$ implies that $m^1_{j'}$ and $m^2_{j'}$ do not have the same parity as the big diffrence values on $T_5,\ldots,T_{2\theta}$. Hence, it follows from (f) that $\pi_i\nrightarrow \pi_j$ for $i$ or $j$ in $J'$, $2\leq j<i\leq \theta$.\\
	 
So, after the flipping of all $\pi_{\ell}$ with $\ell\in J\cup J'$, the remaining conflicts $\pi_i\rightarrow \pi_j$ with $2\leq j <i\leq \theta$ are such that $\{i,j\}\cap (J\cup J')=\emptyset$ . Consider any such conflict. If there is $i'\neq i$ such that  $\pi_{i'}\rightarrow \pi_j$, then we know from (d) that $i'=i-1$ or $i+1$. Without loss of generality, we may assume $i'=i+1$ (else we permute the roles of $i$ and $i'$). Since none of $j,i,i'$ belongs to $J\cup J'$, there is no $k$ such that $\pi_k\rightarrow \pi_i$, $\pi_k\rightarrow \pi_{i'}$ or $\pi_j\rightarrow \pi_k$. Also, it follows from (d) that there is no $k\neq i,i'$ such that $\pi_k\rightarrow \pi_{j}$
\begin {itemize}
\item if $i\leq 2\theta/3$, we flip $\pi_j$. We then have $\min\{|b^1_i|,|b^2_i|\}\geq 6\theta+3t-4(2\theta/3)-1=10\theta/3+3t-1$. It follows that  $j\leq 2\theta/9$ else $\max\{|m^1_j|,|m^2_j|\}\leq 4\theta+3t-3(2\theta)/9-2=10\theta/3+3t-2$. Hence $\min\{|b^1_j|,|b^2_j|\}\geq 6\theta+3t-4(2\theta/9)-1=46\theta/9+3t-1>4\theta+3t-2.$ Since the medium magnitudes are at most equal to $4\theta+3t-2$, we cannot have $\pi_j\rightarrow \pi_k$ after the flip of $\pi_j$. Also, it follows from (g) that, after the flip of $\pi_j$, we have
$\pi_k\nrightarrow \pi_{j}$ for $j<k\leq \theta$. Hence, after the flip of $\pi_j$, the difference labels on its two triangles are different from those on the other triangles $T_k$, $k\geq 5$.
\item if $i> 2\theta/3$, we flip $\pi_i$ and $\pi_{i'}$ (if any). In this case, we have $\max\{|m^1_{i'}|,|m^2_{i'}|\}<\max\{|m^1_i|,|m^2_i|\}$ $\leq 4\theta+3t-3(2\theta/3)=2\theta+3t$. Since all big magnitudes on $T_1,\ldots,T_{2\theta}$ are strictly larger than $2\theta+3t$, we cannot have $\pi_k\rightarrow \pi_i$ after the flip of $\pi_i$ and $\pi_{i'}$. Also, it follows from (f) that after the flip of $\pi_i$ and $\pi_{i'}$,
we have $\pi_i\nrightarrow \pi_{k}$ and $\pi_{i'}\nrightarrow \pi_{k}$ for $2\leq k<i$. Hence, after the flip of $\pi_i$ and $\pi_{i'}$, the difference labels on their triangles are different from those on the other triangles $T_k$, $k\geq 5$.
\end{itemize}

	After all these flips, there is no $\pi_i\rightarrow \pi_j$ with $2\leq j<i\leq \theta$. We consider now triangles $T_1,T_2,T_3,T_4$ involved in $\pi_0$ and $\pi_1$. If there is $j\geq 2$ such that $\pi_j\rightarrow \pi_1$ then we know from (e) that $\pi_j\nrightarrow \pi_k$ for all $2\leq k<j$. Hence, $j\notin J\cup J'$. If, before the flips, there was $i$ such that $\pi_i\rightarrow \pi_j$, then $i> 2\theta/3$. Indeed, we have seen above that if $i\leq 2\theta/3$, then $\min\{|b^1_j|,|b^2_j|\}>4\theta+3t-2$, which means that $\pi_j\nrightarrow \pi_1$. So, $\pi_j$ was not flipped, and by flipping $\pi_1$, we get $\pi_j\nrightarrow \pi_1$ for all $2\leq j\leq \theta$.
		Since the parity of $m^0_1$ and $m^0_2$ is the opposite of the parity of $b^1_i$ and $b^2_i$ for all $i\geq 2$, we have  $\pi_j\nrightarrow \pi_0$ for all $2\leq j\leq \theta$.
	Hence, the only possible remaining conflict is between $\pi_0$ and $\pi_1$. This can only occur if $b^0_1=b^1_1$ and $\pi_1$ was flipped. In such a case, we flip $\pi_0$ to remove this last conflict.\\

{\bf Case B :} $n=2\theta+t$, $\theta=3t+1$

\vspace{0.1cm}We treat this case as the previous one. More precisely, the vertex labels $f(v_i)$ on $T_1,\ldots,T_{2\theta}$ are given in Table \ref{num2}. Given a gdl $f'$ for $t\overrightarrow{\bf C_3}$ with at most one arc of magnitude $3t-2$, and all other arcs of magnitude strictly smaller than $3t-2$, we set $f(v_i)=f'(v_i)+3\theta$ for $i=6\theta+1,\ldots,6\theta+3t$.
Again, one can easily check that $f$ is a bijection between $\{v_1,\ldots,v_{6\theta+3t}\}$ and $\{1,
\ldots,6\theta+3t=3n\}$.

			\begin{table}[H]
				\begin{center}
					\begin{tabular}{|c|c|c|c|}\hline
						Triangle $T_i$ &$f(v_{3i-2})$&$f(v_{3i-1})$&$f(v_{3i})$ \\
						\hline
						$T_1$&$1$&$2\theta+1$&$6\theta+3t-3$ \\
						$T_2$&$2$&$6\theta+3t$&$4\theta+3t$ \\
						$T_3$&$3$&$6\theta+3t-1$&$2\theta+2$ \\
						$T_4$&$4$&$4\theta+3t-1$&$6\theta+3t-2$ \\
						\hdashline
						$\vdots$&$\vdots$&$\vdots$&$\vdots$ \\
						$T_{2k-1}$&$2k-1$&$2\theta+k$&$6\theta+3t-2k+2$ \\
						$T_{2k}$&$2k$&$6\theta+3t-2k+1$&$4\theta+3t-k+1$ \\
						$k=3,\ldots,\theta-1$&&&\\
						$\vdots$&$\vdots$&$\vdots$&$\vdots$ \\
						\hdashline
						$T_{2\theta-1}$&$2\theta-1$&$3\theta+3t+1$&$4\theta+3t+1$ \\
						$T_{2\theta}$&$2\theta$&$4\theta+3t+2$&$3\theta$ \\
						\hline
					\end{tabular}
					\caption{The labeling of $T_1,\ldots,T_{2\theta}$ for case B.}
					\label{num2}
				\end{center}
			\end{table}
			
\vspace{-0.3cm}The small, medium, and big difference labels for triangles $T_1,\ldots,T_{2\theta}$ are given in Table \ref{dl2}. Again, the triangles are grouped in pairs, using two dummy triangles $D_1$ and $D_2$ which are paired with $T_5$ and $T_{2\theta-2}$, respectively. Notice that for every $uv$ on a $T_i$ with $i\leq 2\theta$ and every $u'v'$ on a $T_j$ with $j>2\theta$, we have $f(v)-f(u)\neq f(v')-f(u')$ since the smallest possible magnitude for $uv$ is $\theta=3t+1$, while the largest possible magnitude for $u'v'$ is $3t-2$. Hence, in this case, we do not have to flip triangles $T_{2\theta+1},\ldots,T_{2\theta+t}$. Note also that the largest magnitude is $6\theta+3t-2=	3n-2$, and there is exactly one arc with this magnitude. 

\begin{table}[H]
		\begin{center}
			\begin{tabular}{|c|c|c|c|c|}\hline
				Pair&Triangle&Small-dl&Medium-dl&Big-dl \\
				\hline
				\multirow{2}{3cm}{$\pi_0=(T_1,T_2)$}&$T_1$&$2\theta$&$4\theta+3t-4$&$-(6\theta+3t-4)$ \\
				&$T_2$&$-2\theta$&$-(4\theta+3t-2)$&$6\theta+3t-2$ \\
				\hdashline
				\multirow{2}{3cm}{$\pi_1=(T_3,T_4)$}&$T_3$&$-(2\theta-1)$&$-(4\theta+3t-3)$&$6\theta+3t-4$ \\
				&$T_4$&$2\theta-1$&$4\theta+3t-5$&$-(6\theta+3t-6)$ \\
				\hdashline
				\multirow{2}{3cm}{$\pi_2=(D_1,T_5)$}&$D_1$&$-(2\theta-2)$&$-(4\theta+3t-5)$&$6\theta+3t-7$ \\
				&$T_5$&$2\theta-2$&$4\theta+3t-7$&$-(6\theta+3t-9)$ \\
				\hdashline
				$\vdots$&$\vdots$&$\vdots$&$\vdots$&$\vdots$\\
				$\pi_k=(T_{2k},T_{2k+1})$&$T_{2k}$&$-(2\theta-k)$&$-(4\theta+3t-3k+1)$&$6\theta+3t-4k+1$ \\
				$k=3,\ldots,\theta-2$&$T_{2k+1}$&$2\theta-k$&$4\theta+3t-3k-1$&$-(6\theta+3t-4k-1)$ \\
				$\vdots$&$\vdots$&$\vdots$&$\vdots$&$\vdots$\\
				\hdashline
				\multirow{2}{3.3cm}{$\pi_{\theta-1}=(T_{2\theta-2},D_{2})$}&$T_{2\theta-1}$&$-(\theta+1)$&$-(\theta+3t+4)$&$2\theta+3t+5$ \\
				&$D_{2}$&$\theta+1$&$\theta+3t+2$&$-(2\theta+3t+3)$ \\
				\hdashline
				\multirow{2}{3cm}{$\pi_{\theta}=(T_{2\theta-1},T_{2\theta})$}&$T_{2\theta-1}$&$\theta$&$\theta+3t+2$&$-(2\theta+3t+2)$ \\
				&$T_{2\theta}$&$-\theta$&$-(\theta+3t+2)$&$2\theta+3t+2$ \\
				\hline
			\end{tabular}
			\caption{The difference labels of the arcs of $T_1,\ldots,T_{2\theta},D_1,D_2$ for case B.}
			\label{dl2}
		\end{center}
	\end{table}

\vspace{-0.4cm}Since $\theta=3t+1$, we have $\theta+3t+2=2\theta+1$, which means that no medium-dl can be equal to a small-dl. The small, medium and big difference labels on $T_1,\ldots,T_{2\theta-2}$ are exactly the same as those of Table \ref{dl1}. Using the same arguments, as in the previous case, we can avoid conflicts involving medium and big difference labels of $\pi_0,\ldots,\pi_{\theta-1}$. 
Consider now $\pi_{\theta}$:
\begin{itemize}
\vspace{-0.4cm}\item the medium difference values of $\pi_{\theta}$ can only be conflicting with the medium-dl of $D_2$, but we don't care about such a conflict since $D_2$ is a dummy triangle;
	\vspace{-0.2cm}\item the big difference values of $\pi_{\theta}$ can only be conflicting with the medium-dl of a $T_k$. For this to happen, we should have $2\theta+3t+2$ equal to $4\theta+3t-3k+1$ or $4\theta+3t-3k-1$, or equivalently $k$ equal to $\frac{2\theta-1}{3}=\frac{6t+1}{3}$ or $\frac{2\theta-3}{3}=\frac{6t-1}{3}$, which is impossible since $k$ is an integer.
\end{itemize}

{\bf Case C :} $n=2\theta+t-1$, $\theta\in\{3t-1,3t\}$

\vspace{0.1cm}Again, consider the vertex labels $f(v_i)$ on $T_1,\ldots,T_{2\theta-1}$ shown in Table \ref{num3}. Given a gdl $f'$ for $t\overrightarrow{\bf C_3}$ with at most one arc of magnitude $3t-2$, and all other arcs of magnitude strictly smaller than $3t-2$, we set $f(v_i)=f'(v_i)+3\theta-1$ for $i=6\theta-2,\ldots,6\theta+3t-3$. One can easily check $f$ is a bijection between $\{v_1,\ldots,v_{6\theta+3t-3}\}$ and $\{1,
\ldots,6\theta+3t-3=3n\}$. 
The small, medium, and big difference labels for triangles $T_1,\ldots,T_{2\theta}$ are given in Table \ref{dl3}. 

\begin{table}[H]
	\begin{center}
		\begin{tabular}{|c|c|c|c|}\hline
			Triangle $T_i$ &$f(v_{3i-2})$&$f(v_{3i-1})$&$f(v_{3i})$ \\
			\hline
			$T_1$&$1$&$2\theta$&$6\theta+3t-6$ \\
			$T_2$&$2$&$6\theta+3t-3$&$4\theta+3t-2$ \\
			$T_3$&$3$&$6\theta+3t-4$&$2\theta+1$ \\
			$T_4$&$4$&$4\theta+3t-3$&$6\theta+3t-5$ \\
			$T_5$&$5$&$2\theta+2$&$6\theta+3t-7$ \\
			\hdashline
			$\vdots$&$\vdots$&$\vdots$&$\vdots$ \\
			$T_{2k}$&$2k$&$6\theta+3t-2k-2$&$4\theta+3t-k-1$ \\
			$T_{2k+1}$&$2k+1$&$2\theta+k$&$6\theta+3t-2k-3$ \\
			$k=3,\ldots,\theta-1$&&&\\
			$\vdots$&$\vdots$&$\vdots$&$\vdots$ \\
			\hline
		\end{tabular}
		\caption{The labeling of $T_1,\ldots,T_{2\theta-1}$ for case C.}
		\label{num3}
	\end{center}
\end{table}

\begin{table}[H]
						\begin{center}
							\begin{tabular}{|c|c|c|c|c|}\hline
								Pair&Triangle&Small-dl&Medium-dl&Big-dl \\
								\hline
								\multirow{2}{3cm}{$\pi_0=(T_1,T_2)$}&$T_1$&$2\theta-1$&$4\theta+3t-6$&$-(6\theta+3t-7)$ \\
								&$T_2$&$-(2\theta-1)$&$-(4\theta+3t-4)$&$6\theta+3t-5$ \\
								\hdashline
								\multirow{2}{3cm}{$\pi_1=(T_3,T_4)$}&$T_3$&$-(2\theta-2)$&$-(4\theta+3t-5)$&$6\theta+3t-7$ \\
								&$T_4$&$2\theta-2$&$4\theta+3t-7$&$-(6\theta+3t-9)$ \\
								\hdashline
								\multirow{2}{3cm}{$\pi_2=(D_1,T_5)$}&$D_1$&$-(2\theta-3)$&$-(4\theta+3t-7)$&$6\theta+3t-10$ \\
								&$T_5$&$2\theta-3$&$4\theta+3t-9$&$-(6\theta+3t-12)$ \\
								\hdashline
								$\vdots$&$\vdots$&$\vdots$&$\vdots$&$\vdots$\\
								$\pi_k=(T_{2k},T_{2k+1})$&$T_{2k}$&$-(2\theta-k-1)$&$-(4\theta+3t-3k-1)$&$6\theta+3t-4k-2$ \\
								$k=3,\ldots,\theta-1$&$T_{2k+1}$&$2\theta-k-1$&$4\theta+3t-3k-3$&$-(6\theta+3t-4k-4)$ \\
								$\vdots$&$\vdots$&$\vdots$&$\vdots$&$\vdots$\\
								\hline
							\end{tabular}
							\caption{The difference labels of the arcs of $T_1,\ldots,T_{2\theta-1},D_1$ for case C.}
							\label{dl3}
						\end{center}
					\end{table}

Again, the triangles are grouped in pairs, using one dummy triangle $D_1$ which is paired with $T_5$. Notice that for every $uv$ on a $T_i$ with $i\leq 2\theta-1$ and every $u'v'$ on a $T_j$ with $j>2\theta-1$, we have $f(v)-f(u)\neq f(v')-f(u')$ since the smallest possible magnitude for $uv$ is $\theta\geq 3t-1$, while the largest possible magnitude for $u'v'$ is $3t-2$. Hence, also in this case, we do not have to flip $T_{2\theta+1},\ldots,T_{2\theta+t}$. Note also that the largest magnitude is $6\theta+3t-5=3n-2$, and there is exactly one arc with this magnitude. \\

\vspace{-0.1cm}Since $\theta<3t+1$, we have $\theta+3t>2\theta-1$, which means that no medium-dl can be equal to a small-dl. Using the same arguments, as in the previous cases, we can avoid conflicts involving $\pi_2,\ldots,\pi_{\theta-1}$. \\
			
\vspace{-0.1cm}			If there is $j\geq 2$ such that $\pi_j\rightarrow \pi_0$, then assume there is $i>j$ such that $\pi_i\rightarrow \pi_j$. If $i\leq 2\theta/3$, then $\min\{|b^1_i|,|b^2_i|\}\geq 6\theta+3t-4(2\theta/3)-4=10\theta/3+3t-4$. It follows that  $j\leq (2\theta+3)/9$ else $\max\{|m^1_j|,|m^2_j|\}\leq 4\theta+3t-3(2\theta+3)/9-4=10\theta/3+3t-5$. Hence $\min\{|b^1_j|,|b^2_j|\}\geq 6\theta+3t-4(2\theta+3)/9-4=46\theta/9+3t-48/9>4\theta+3t-4$, which contradicts $\pi_j\rightarrow \pi_0$.  
			Hence, we necessarily have $i>2\theta/3$, and since $j$ cannot belong to $J\cup J'$, we conclude that $j$ was not flipped. Hence, by flipping $\pi_0$, we get $\pi_j\nrightarrow \pi_0$ for all $j\geq 2$.\\

\vspace{-0.1cm}		Since the parity of $m^1_1$ and $m^2_1$ is the opposite of the parity of $b^1_i$ and $b^2_i$ for all $i\geq 2$, we have  $\pi_j\nrightarrow \pi_1$ for all $j\geq 2$.
			Hence, the only possible remaining conflict is between $\pi_0$ and $\pi_1$. This can only occur if $b^1_0=b^1_1$ and $\pi_1$ was flipped. In such a case, we flip $\pi_1$ to remove this last conflict.\\

{\bf Case D :} $n=2\theta+t-1$, $\theta=3t+1$
	
\vspace{0.1cm}Consider the vertex labels $f(v_i)$ on $T_1,\ldots,T_{2\theta-1}$ shown in Table \ref{num4}. Given a gdl $f'$ for $t\overrightarrow{\bf C_3}$ with at most one arc of magnitude $3t-2$, and all other arcs of magnitude strictly smaller than $3t-2$, we set $f(v_i)=f'(v_i)+3\theta-1$ for $i=6\theta-2,\ldots,6\theta+3t-3$. One can easily check $f$ is a bijection between $\{v_1,\ldots,v_{6\theta+3t-3}\}$ and $\{1,
	\ldots,6\theta+3t-3=3n\}$.

\vspace{-0.1cm}	\begin{table}[H]
		\begin{center}
			\begin{tabular}{|c|c|c|c|}\hline
				Triangle $T_i$ &$f(v_{3i-2})$&$f(v_{3i-1})$&$f(v_{3i})$ \\
				\hline
				$T_1$&$1$&$2\theta$&$6\theta+3t-6$ \\
				$T_2$&$2$&$6\theta+3t-3$&$4\theta+3t-2$ \\
				$T_3$&$3$&$6\theta+3t-4$&$2\theta+1$ \\
				$T_4$&$4$&$4\theta+3t-3$&$6\theta+3t-5$ \\
				$T_5$&$5$&$2\theta+2$&$6\theta+3t-7$ \\
				\hdashline
				$\vdots$&$\vdots$&$\vdots$&$\vdots$ \\
				$T_{2k}$&$2k$&$6\theta+3t-2k-2$&$4\theta+3t-k-1$ \\
				$T_{2k+1}$&$2k+1$&$2\theta+k$&$6\theta+3t-2k-3$ \\
				$k=3,\ldots,\theta-3$&&&\\
				$\vdots$&$\vdots$&$\vdots$&$\vdots$ \\
				\hdashline
				$T_{2\theta-4}$&$2\theta-4$&$4\theta+3t-1$&$3\theta+3t+1$ \\
				$T_{2\theta-3}$&$2\theta-3$&$4\theta+3t+2$&$3\theta-2$ \\
				$T_{2\theta-2}$&$2\theta-2$&$3\theta+3t$&$4\theta+3t+1$ \\
				$T_{2\theta-1}$&$2\theta-1$&$3\theta-1$&$4\theta+3t$ \\
				\hline
			\end{tabular}
			\caption{The labeling of $T_1,\ldots,T_{2\theta-1}$ for case D.}
			\label{num4}
		\end{center}
	\end{table}

 \renewcommand{\arraystretch}{1.0}. \setlength{\tabcolsep}{0.05cm}
 \begin{table}[H]
 	\begin{center}
 		\begin{tabular}{|c|c|c|c|c|}\hline
 			Pair&Triangle&Small-dl&Medium-dl&Big-dl \\
 			\hline
 			\multirow{2}{3cm}{$\pi_0=(T_1,T_2)$}&$T_1$&$2\theta-1$&$4\theta+3t-6$&$-(6\theta+3t-7)$ \\
 			&$T_2$&$-(2\theta-1)$&$-(4\theta+3t-4)$&$6\theta+3t-5$ \\
 			\hdashline
 			\multirow{2}{3cm}{$\pi_1=(T_3,T_4)$}&$T_3$&$-(2\theta-2)$&$-(4\theta+3t-5)$&$6\theta+3t-7$ \\
 			&$T_4$&$2\theta-2$&$4\theta+3t-7$&$-(6\theta+3t-9)$ \\
 			\hdashline
 			\multirow{2}{3cm}{$\pi_2=(D_1,T_5)$}&$D_1$&$-(2\theta-3)$&$-(4\theta+3t-7)$&$6\theta+3t-10$ \\
 			&$T_5$&$2\theta-3$&$4\theta+3t-9$&$-(6\theta+3t-12)$ \\
 			\hdashline
 			$\vdots$&$\vdots$&$\vdots$&$\vdots$&$\vdots$\\
 			$\pi_k=(T_{2k},T_{2k+1})$&$T_{2k}$&$-(2\theta-k-1)$&$-(4\theta+3t-3k-1)$&$6\theta+3t-4k-2$ \\
 			$k=3,\ldots,\theta-3$&$T_{2k+1}$&$2\theta-k-1$&$4\theta+3t-3k-3$&$-(6\theta+3t-4k-4)$ \\
 			$\vdots$&$\vdots$&$\vdots$&$\vdots$&$\vdots$\\
 			\hdashline
 			\multirow{2}{3.9cm}{$\pi_{\theta-2}=(T_{2\theta-3},T_{2\theta-2})$}&$T_{2\theta-3}$&$-(\theta+1)$&$-(\theta+3t+4)$&$2\theta+3t+5$ \\
 			&$T_{2\theta-2}$&$\theta+1$&$\theta+3t+2$&$-(2\theta+3t+3)$ \\
 			\hdashline
 			\multirow{2}{3.9cm}{$\pi_{\theta-1}=(T_{2\theta-1},T_{2\theta-4})$}&$T_{2\theta-1}$&$\theta$&$\theta+3t+1$&$-(2\theta+3t+1)$ \\
 			&$T_{2\theta-4}$&$-(\theta-2)$&$-(\theta+3t+5)$&$2\theta+3t+3$ \\
 			\hline
 		\end{tabular}
 		\caption{The difference labels of the arcs of $T_1,\ldots,T_{2\theta-1},D_1$ for case D.}
 		\label{dl4}
 	\end{center}
 \end{table}

The small, medium, and big difference labels for triangles $T_1,\ldots,T_{2\theta}$ are given in Table \ref{dl4}. Again, the triangles are grouped in pairs, using one dummy triangle $D_1$ which is paired with $T_5$. Notice that for every $uv$ on a $T_i$ with $i\leq 2\theta-1$ and every $u'v'$ on a $T_j$ with $j>2\theta-1$, we have $f(v)-f(u)\neq f(v')-f(u')$ since the smallest possible magnitude for $uv$ is $\theta-2=3t-1$, while the largest possible magnitude for $u'v'$ is $3t-2$. Hence, also in this case, we do not have to flip $T_{2\theta+1},\ldots,T_{2\theta+t}$. Note also that the largest magnitude is $6\theta+3t-5=3n-2$, and there is only one arc with this magnitude.\\ 

\vspace{-0.1cm}Since $\theta=3t+1$, we have $\theta+3t+1=2\theta$, which means that no medium-dl can be equal to a small-dl. The small, medium and big difference labels on $T_1,\ldots,T_{2\theta-5}$ are exactly the same as those of Table \ref{dl3}. Using the same arguments, as in the previous case, we can avoid conflicts involving $\pi_0,\ldots,\pi_{\theta-3}$.\\

	\vspace{-0.1cm}Consider now $\pi_{\theta-2}$ and $\pi_{\theta-1}$. The medium magnitudes $|m^1_{\theta-2}|,|m^2_{\theta-2}|,|m^1_{\theta-1}|$ and $|m^2_{\theta-1}|$ do not appear on any other triangle. Also, the medium magnitudes on a $\pi_k$ with $2\leq k\leq \theta-3$ are equal to $4\theta+3t-3k-1=15t-3k+3$ or $4\theta+3t-3k-3=15t-3k+1$, which mean that they are all equal to $0$, or $1\mod 3$. Hence, the big magnitudes $|b^2_{\theta-2}|=|b^2_{\theta-1}|=2\theta+3t+3=9t+5$ do not appear on any other triangle as medium magnitude. Therefore, these two big magnitudes will not be conflicting if we either flip both $\pi_{\theta-1}$ and $\pi_{\theta-2}$, or none of them. The only remaining possible conflicts involve a medium-dl on a $T_i$ ($i<\theta-2$) and $b^1_{\theta-2}$ or $b^1_{\theta-1}$\\
	
	\vspace{-0.1cm}Assume there is a triangle $T_i$ with magnitude $2\theta+3t+1=|b^1_{\theta-1}|$. This means that $2\theta+3t+1\leq 4\theta+3t-3i-1$, which is equivalent to $i\leq (2\theta-2)/3$. Hence, $\pi_i$ was not flipped. Also, if there is a triangle $T_j$ with  magnitude $2\theta+3t+5=b^1_{\theta-2}$, then $j<i\leq (2\theta-2)/3$, which means that $\pi_j$ was not flipped. Now, 
	\begin{itemize}
		\item if there is a triangle $T_i$ with medium-dl $-(2\theta+3t+1)$, then $m^1_{i}=b^1_{\theta-1}$, and $m^2_{i-2}=-b^1_{\theta-1}+4=2\theta+3t+5=b^1_{\theta-2}$, and we can avoid both conflicts by flipping both $\pi_{\theta-1}$ and $\pi_{\theta-2}$; 
	\item if there is a triangle $T_j$ with medium-dl $2\theta+3t+5$, then $m^2_{j}=b^1_{\theta-2}$, and $m^1_{j+2}=-b^1_{\theta-2}+4=-(2\theta+3t+1)=b^1_{\theta-1}$, and we can avoid both conflicts by flipping both $\pi_{\theta-1}$ and $\pi_{\theta-2}$. 
	\item if there is no triangle with medium-dl $-(2\theta+3t+1)$ or $2\theta+3t+5$, there is no conflict.
	\end{itemize}
\vspace{-0.5cm}\end{proof}

We already know from Lemma \ref{C2k+1C4} that $\overrightarrow{\bf C_4}+\overrightarrow{\bf C_3}$ has a gdl. We now show that this is also the case for $\overrightarrow{\bf C_4}+n\overrightarrow{\bf C_3}$, $n\geq 2$.

\begin{lemma}\label{C4C3}
	$\overrightarrow{\bf C_4}+n\overrightarrow{\bf C_3}$ has a gdl for every $n\ge 1$.
\end{lemma}
\begin{proof}
	The graphs in Figures \ref{2C3C4},\ldots,  \ref{8C3C4}  show the existence of the desired gdl for $2\leq n \leq 8$. 
\vspace{0.5cm}\begin{figure}[H]
	\centering
	\includegraphics[width=10.5cm,keepaspectratio=true]{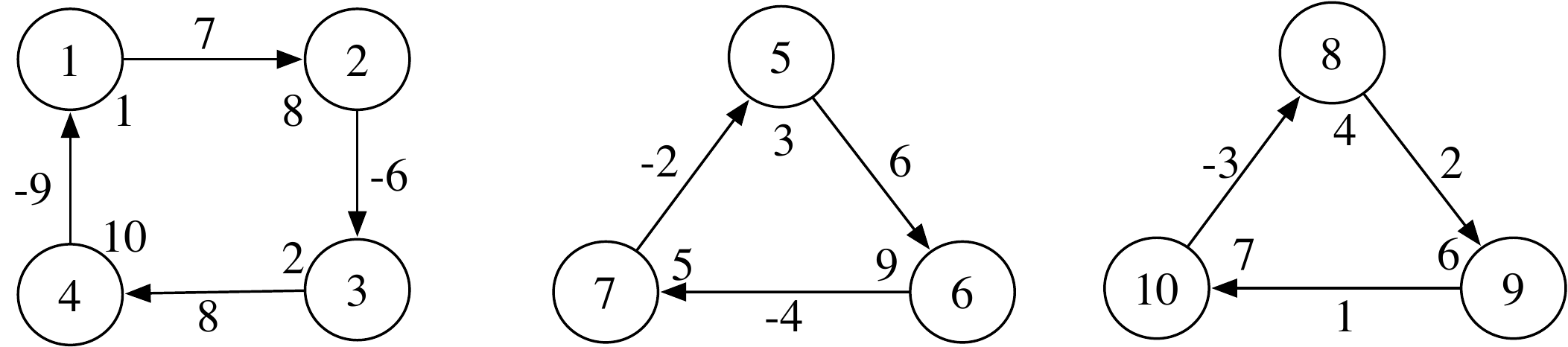}
	\caption{$2\protect\overrightarrow{\bf C_3}+\protect\overrightarrow{\bf C_4}$.}	
	\label{2C3C4}
\end{figure}

\vspace{0.5cm}\begin{figure}[H]
	\centering
	\includegraphics[width=14cm,keepaspectratio=true]{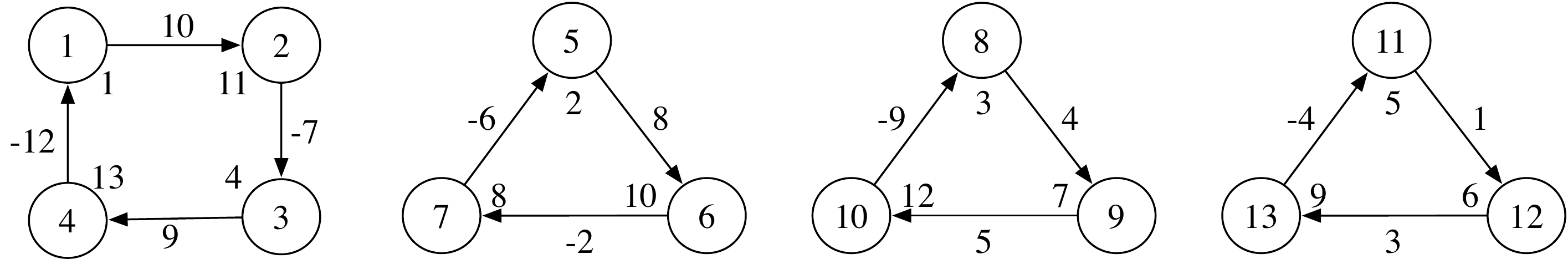}
	\caption{$3\protect\overrightarrow{\bf C_3}+\protect\overrightarrow{\bf C_4}$.}	
	\label{3C3C4}
\end{figure}

\vspace{0.5cm}\begin{figure}[H]
	\centering
	\includegraphics[width=10.5cm,keepaspectratio=true]{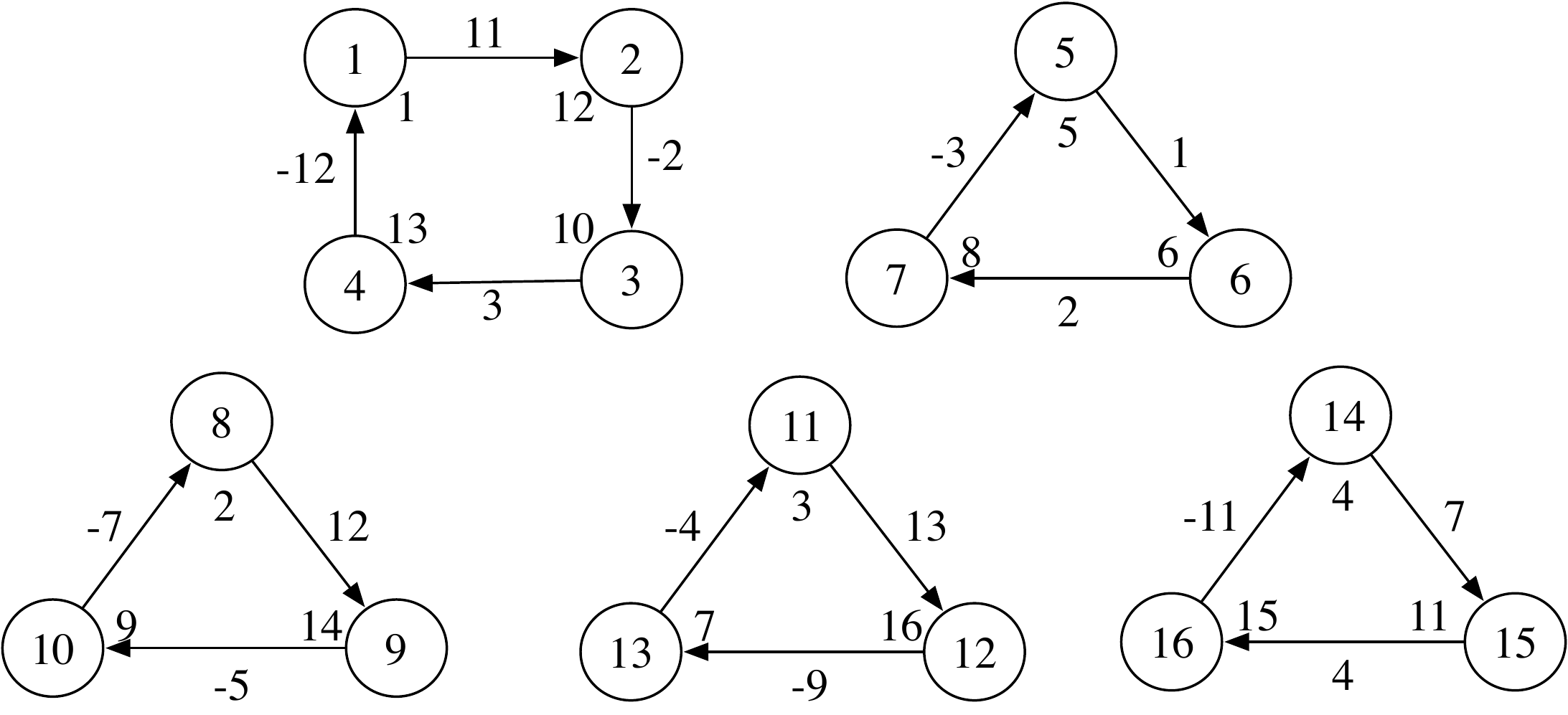}
	\caption{$4\protect\overrightarrow{\bf C_3}+\protect\overrightarrow{\bf C_4}$.}	
	\label{4C3C4}
\end{figure}

\vspace{0.5cm}\begin{figure}[H]
	\centering
	\includegraphics[width=10.5cm,keepaspectratio=true]{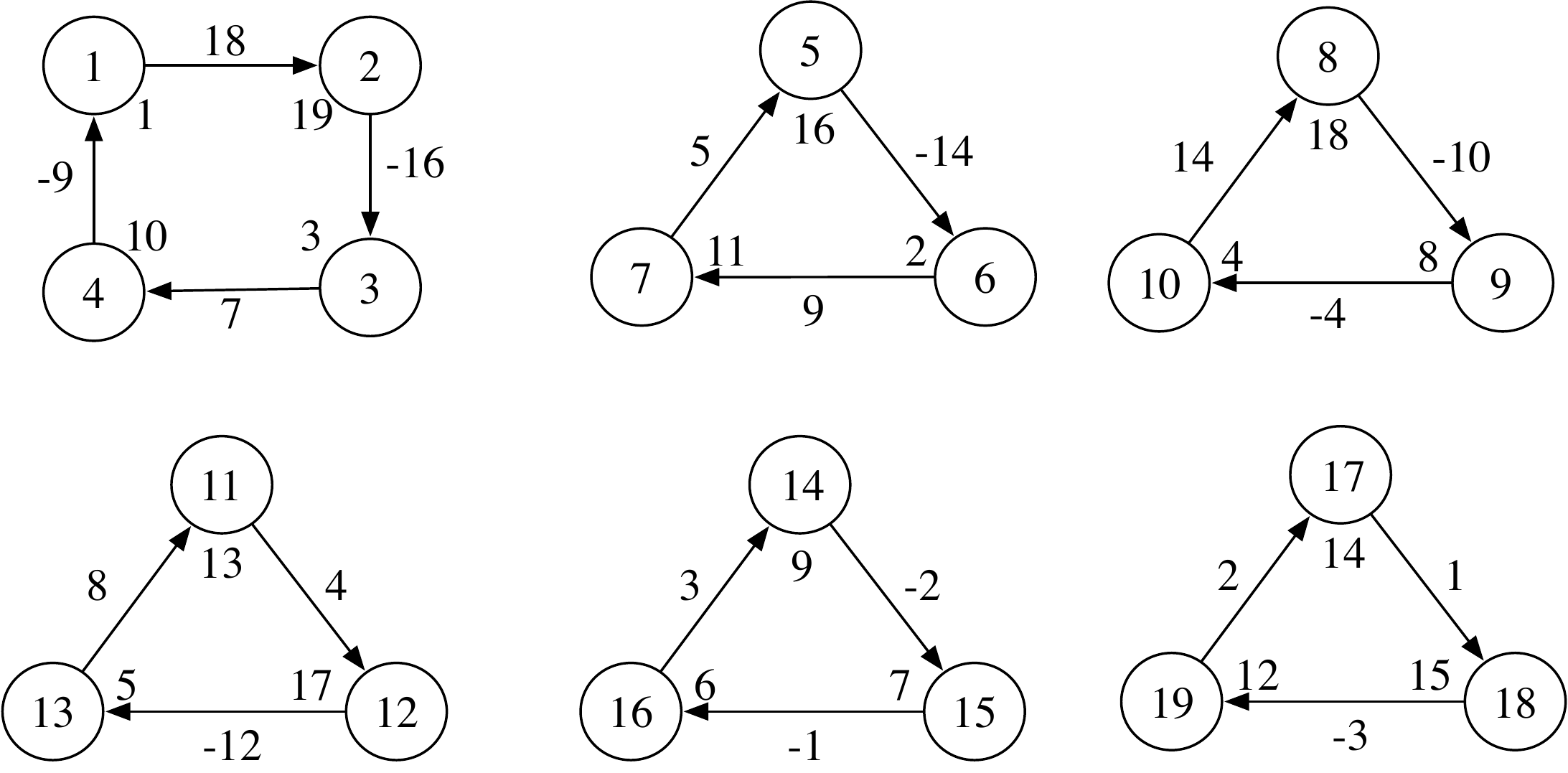}
	\caption{$5\protect\overrightarrow{\bf C_3}+\protect\overrightarrow{\bf C_4}$.}	
	\label{5C3C4}
\end{figure}

\vspace{0.5cm}\begin{figure}[H]
	\centering
	\includegraphics[width=14cm,keepaspectratio=true]{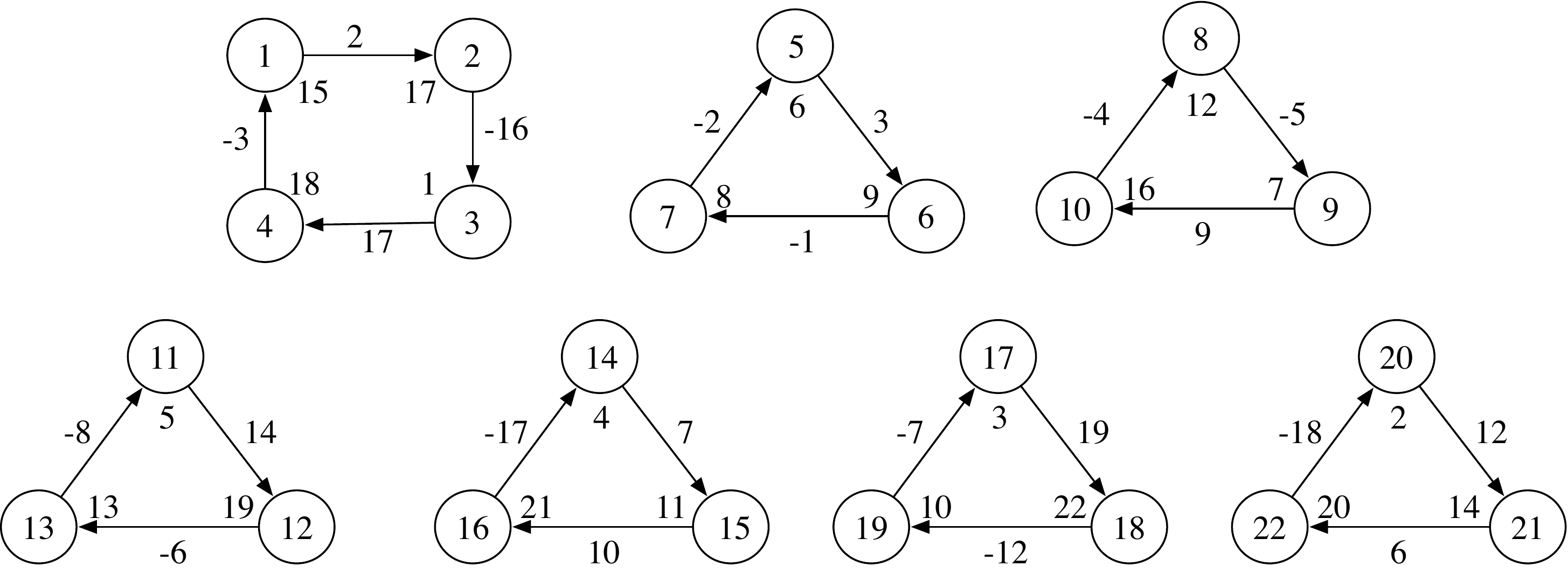}
	\caption{$6\protect\overrightarrow{\bf C_3}+\protect\overrightarrow{\bf C_4}$.}	
	\label{6C3C4}
\end{figure}

\vspace{0.5cm}\begin{figure}[H]
	\centering
	\includegraphics[width=14cm,keepaspectratio=true]{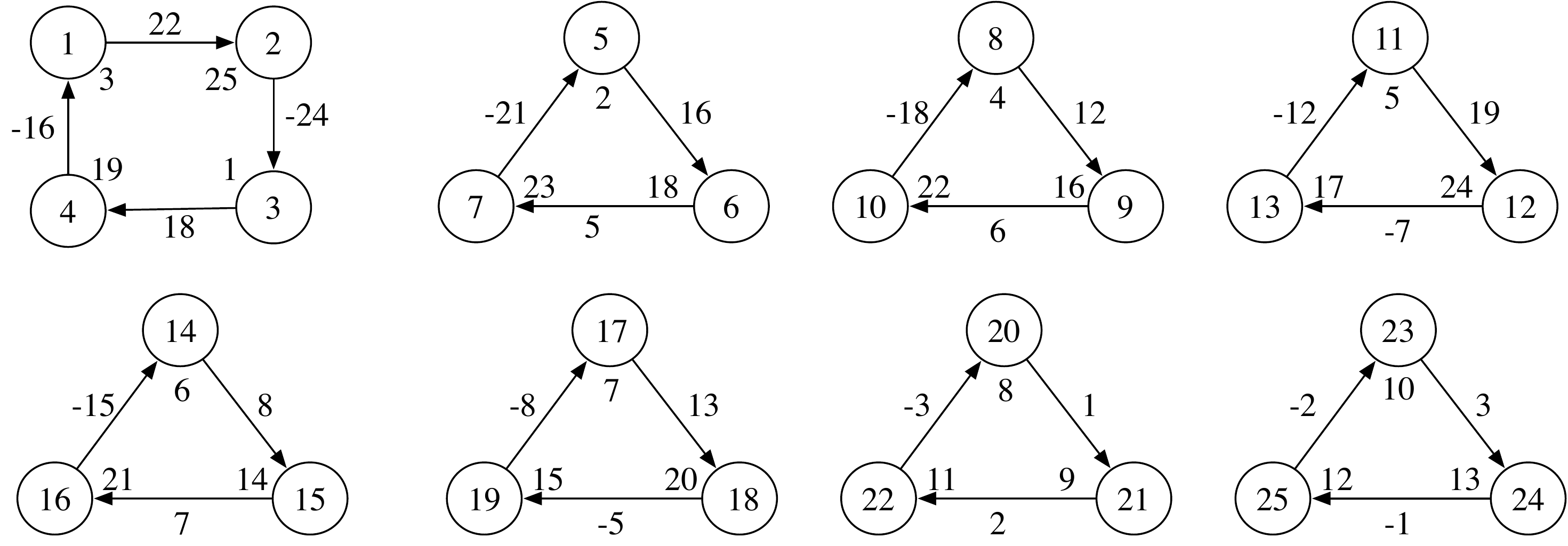}
	\caption{$7\protect\overrightarrow{\bf C_3}+\protect\overrightarrow{\bf C_4}$.}	
	\label{7C3C4}
\end{figure}

\vspace{0.5cm}\begin{figure}[H]
	\centering
	\includegraphics[width=10.5cm,keepaspectratio=true]{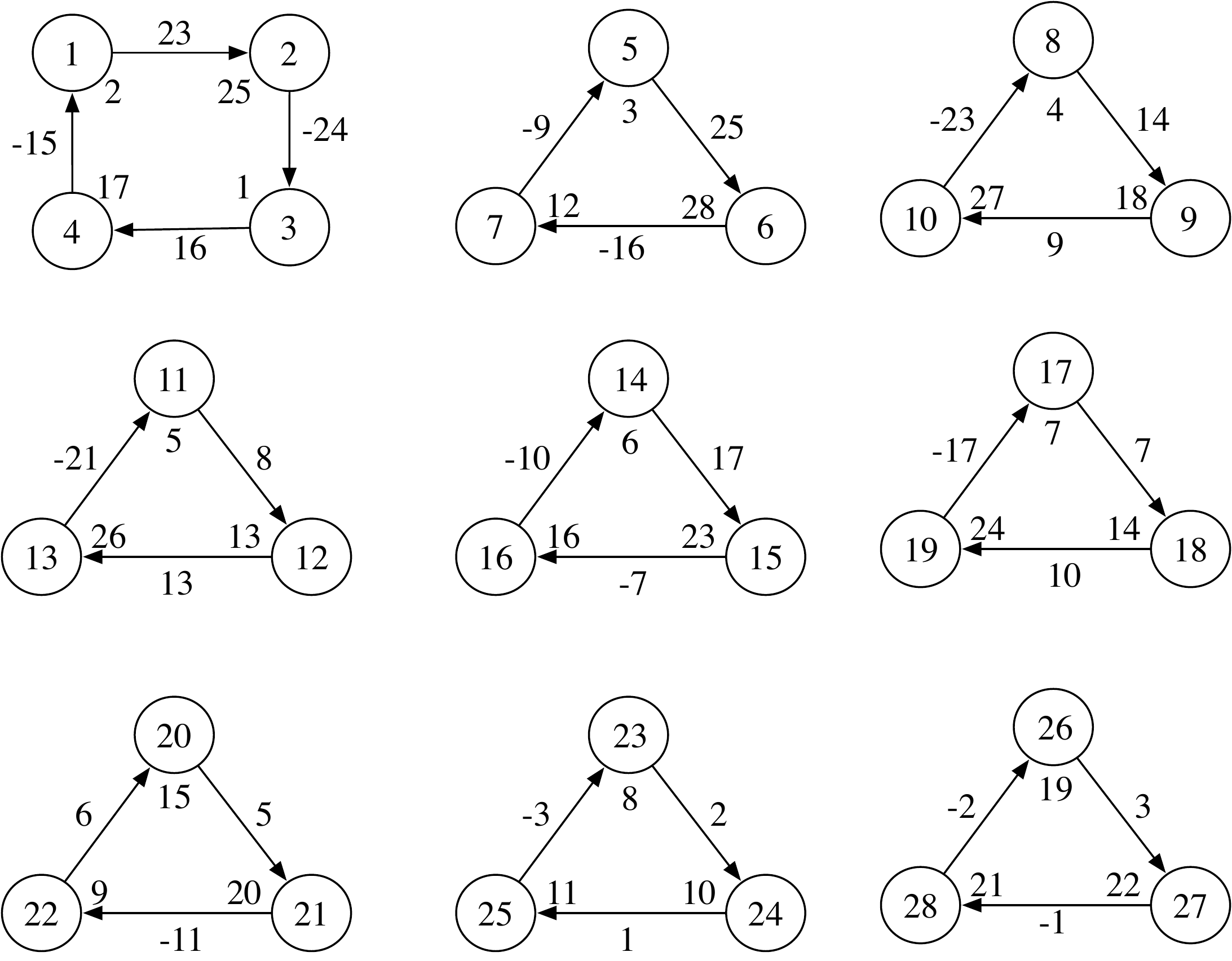}
	\caption{$8\protect\overrightarrow{\bf C_3}+\protect\overrightarrow{\bf C_4}$.}	
	\label{8C3C4}
\end{figure}

For $n\ge 9$, we know from Lemma \ref{nC3} that there is a gdl for $(n+1)\overrightarrow{\bf C_3}$, which can be obtained by performing a set $F$ of flips, starting from the labelling $f$ defined in Tables \ref{num1}, \ref{num2}, \ref{num3}, and \ref{num4} for cases A, B, C and D, respectively. We distinguish two cases.
	\begin{itemize} 
		\item For cases A and B, we consider the graph $G$ obtained from $(n+1)\overrightarrow{\bf C_3}$ by inserting a new vertex $v_0$ between $v_5$ and $v_6$. More precisely, $G$ is obtained by replacing $T_2$ in $(n+1)\overrightarrow{\bf C_3}$ by a $\overrightarrow{\bf C_4}$ with vertex set $\{v_0,v_4,v_5,v_6\}$  and arc set $\{v_4v_5,v_5v_0,v_0v_6,v_6v_4\}$. 
		We then define $f'$ by setting $f'(v_0)=1$ and $f'(v_i)=f(v_i)+1$ for $i=1,\ldots,3(n+1)$. Clearly, $f'$ is bijection between $\{v_0,\ldots,v_{3(n+1)}\}$ and $\{1,\ldots,3n+4\}$. In order to prove that by performing exactly the same set $F$ of flips, we get a gdl for $G$, it is sufficient to show that the difference labels on $v_5v_0$ and $v_0v_6$ cannot appear on other arcs of $G$.
		\begin{itemize}
			\item $\vert f'(v_0)-f'(v_5)\vert=\vert 1-(6\theta+3t+1)\vert =6\theta+3t$, which means that $v_5v_0$ has a magnitude larger than that of any other arc in $G$.
			\item $f'(v_6)-f'(v_0)=(4\theta+3t+1)-1=4\theta+3t$. Since this value is strictly larger than any other medium magnitude in $G$, the difference label on $v_0v_6$ can only be conflicting with a big-dl on a $T_i$ with $i\geq 5$. But this does not occur since these big difference labels have the opposite parity of $4\theta+3t$.  
		\end{itemize}
	\item For cases C and D, we consider the graph $G$ obtained from $(n+1)\overrightarrow{\bf C_3}$ by inserting a new vertex $v_0$ between $v_9$ and $v_7$. More precisely, $G$ is obtained by replacing $T_3$ in $(n+1)\overrightarrow{\bf C_3}$ by a $\overrightarrow{\bf C_4}$ with vertex set $\{v_0,v_7,v_8,v_9\}$  and arc set $\{v_7v_8,v_8v_9,v_9v_0,v_0v_7\}$. We then define $f'$ by setting $f'(v_0)=3n+4=6\theta+3t-2$ and $f'(v_i)=f(v_i)$ for $i=1,\ldots,3(n+1)$. Clearly, $f'$ is bijection between $\{v_0,\ldots,v_{3(n+1)}\}$ and $\{1,\ldots,3n+4\}$. In order to prove that by performing exactly the same set $F$ of flips, we get a gdl for $G$, it is sufficient to show that the difference labels on $v_0v_7$ and $v_9v_0$  do not appear on other arcs of $G$.
	\begin{itemize}
		\item $f'(v_7)-f'(v_0)=3-(6\theta+3t-2) =-(6\theta+3t-5)$. The same difference label appears on $T_2$ but with an opposite sign. These two arcs could be conflicting if exaclty one of $\pi_{0}$ and $\pi_{1}$ is flipped, but this does not occur since $T_1$ and $T_3$ have big difference labels of the same magnitude, but with opposite signs. 
		\item $f'(v_0)-f'(v_9)=(6\theta+3t-2)-(2\theta+1)=4\theta+3t-3$. Since this value is strictly larger than any other medium magnitude in $G$, the difference label on $v_9v_0$ can only be conflicting with a big-dl on a $T_i$ with $i\geq 5$. But this does not occur since these big difference labels have the opposite parity of $4\theta+3t-3$.
	\end{itemize}
	\end{itemize}
\end{proof}

All together, the results shown in the eight lemmas of this section can be summarized as follows.
	
\begin{theorem}
If $G$ is the disjoint union of circuits, among which at most one has an odd length, or all circuits of odd length have 3 vertices, then $G$ has a gdl, unless $G=\overrightarrow{\bf C_{3}}$ or $G=\overrightarrow{\bf C_{2}}+\overrightarrow{\bf C_{3}}$.
\end{theorem}

\section{Conclusion}

As mentioned in the introduction, it is an open question to determine the values of $n$ for which $n\overrightarrow{\bf C_{3}}$ has a graceful labeling, i.e., an injection $f:V\rightarrow\{0,1,\ldots,q\}$  such that, when each arc $xy$ is assigned the label $(f(y)-f(x))\ (mod\ q+1)$, the resulting arc labels are distinct. Considering graceful difference labelings, we could show that  $n\overrightarrow{\bf C_{3}}$ has a gdl if and only if $n\geq 2$. We have also proved additional cases that support the following conjecture. 
\begin{conjecture}
	If $G$ is the disjoint union of circuits, then $G$ has a gdl, unless $G=\overrightarrow{\bf C_{3}}$ or $G=\overrightarrow{\bf C_{2}}+\overrightarrow{\bf C_{3}}$.
\end{conjecture}


\begin{thebibliography}{99}



\bibitem{BTT11}
{\sc J.A.  Gallian}, 
A Dynamic Survey of Graph Labeling, 
{\it The Electronic Journal of Combinatorics} 19(17) (2017) DS 6.
  
 \bibitem{Feng}
{\sc W. Feng,  C. Xu, Jirimutu}, 
A  Survey of the Gracefulness of Diraphs, 
{\it International Journal of Pure and Applied Mathematics} 69(3) (2011) 245-253. 

\bibitem{Golomb}
{\sc S.W. Golomb}, 
How to number a graph, 
{\it Graph Theory and Computing}, R C. Read, ed., Academic Press, New York (1972) 23-37.

\bibitem{Rosa}
{\sc A. Rosa}, 
On certain valuations of the vertices of a graph,
{\it Theory of Graphs (International Symposium, Rome, July 1966)}, 
Gordon and Breach, N.Y. and Dunod Paris (1967) 349-
355.
  
\bibitem{West}
D. B. West,  \textit{Introduction to Graph Theory}, Prentice Hall, (1996). 
\end{thebibliography}
\end{document}